\documentclass[12pt,a4]{amsart}
\usepackage[a4paper, left=28mm, right=28mm, top=28mm, bottom=34mm]{geometry}
\allowdisplaybreaks

\usepackage{amssymb
,amsthm
,amsmath
,amscd
,mathtools	
,mathdots
,leftidx
,color
}

\usepackage[all]{xy}
\usepackage{url}
\usepackage[dvipdfmx]{graphicx}

\AtBeginDocument{%
  \def\MR#1{}
}

\newcommand{\SL}{\mathrm{SL}}
\newcommand{\GL}{\mathrm{GL}}
\newcommand{\Mat}{\mathrm{Mat}}
\newcommand{\SO}{\mathrm{SO}}
\newcommand{\algO}{\mathrm{O}}
\newcommand{\GSp}{\mathrm{GSp}}

\newcommand{\GSpin}{\mathrm{GSpin}}

\newcommand{\oo}{\mathcal{O}}
\newcommand{\p}{\mathfrak{p}}

\newcommand\Z{\mathbb{Z}}
\newcommand\A{\mathbb{A}}
\newcommand\F{\mathbb{F}}
\newcommand\Q{\mathbb{Q}}
\newcommand\R{\mathbb{R}}
\newcommand\C{\mathbb{C}}
\newcommand\dt{\mathbb{S}}
\newcommand\K {\mathbb{K}}
\newcommand{\plim}[1][]{\mathop{\varprojlim}\limits_{#1}}

\newcommand\et{\textup{\'et}}
\newcommand\e{\textup{\'e}}
\newcommand\End{\textup{End}}
\newcommand\Aut{\textup{Aut}}

\newcommand\Ker{\textup{Ker}}
\newcommand\Gal{\mathrm{Gal}}
\newcommand\Shaf{\mathrm{Shaf}}

\newcommand\Sh{\mathrm{Sh}}
\newcommand\Pic{\mathrm{Pic}}
\newcommand\disc{\mathrm{disc}}
\newcommand\Spec{\mathop{\mathrm{Spec}}}
\newcommand\Frac{\mathrm{Frac}}
\newcommand\shf{\mathrm{shf}}
\newcommand\Tw{\mathrm{Tw}}
\newcommand\ch{\mathrm{ch}}
\newcommand\spp{\mathrm{sp}}
\newcommand\tr{\mathrm{tr}}

\newcommand\la{\mathcal{L}}

\newcommand\M{\mathrm{M}}
\newcommand\W{\mathrm{W}}
\newcommand\gr{\mathrm{gr}}

\theoremstyle{plain}
\newtheorem{theorem}[subsubsection]{Theorem}
\newtheorem{prop}[subsubsection]{Proposition}
\newtheorem{lemma}[subsubsection]{Lemma}

\newtheorem{corollary}[subsubsection]{Corollary}

\theoremstyle{definition}

\newtheorem{definition}[subsubsection]{Definition}
\newtheorem{remark}[subsubsection]{Remark}

\newtheorem*{lemma*}{Lemma}
\newtheorem*{prop*}{Proposition}
\newtheorem*{theorem*}{Theorem}
\newtheorem*{claim*}{Claim}
\newtheorem{definition*}{Definition}

\begin{document}

\title[Cohomological Shafarevich conjecture for K3 surfaces]{On a cohomological generalization of the Shafarevich conjecture for K3 surfaces}
\author[T.\ Takamatsu]{Teppei Takamatsu}

\date{\today}


\address{Graduate School of Mathematical Sciences, The University of Tokyo, 3-8-1 Komaba, Meguro-ku, Tokyo 153-8914, Japan}
\email{teppei@ms.u-tokyo.ac.jp}

\maketitle
\begin{abstract}
The Shafarevich conjecture for K3 surfaces asserts the finiteness of isomorphism classes of K3 surfaces over a fixed number field admitting good reduction away from a fixed finite set of finite places. Andr$\e$ proved this conjecture for polarized K3 surfaces of fixed degree, and recently She proved it for polarized K3 surfaces of unspecified degree. In this paper, we prove a certain generalization of their results, which is stated by the unramifiedness of $\ell$-adic $\e$tale cohomology groups for K3 surfaces over finitely generated fields of characteristic $0$. As a corollary, we get the original Shafarevich conjecture for K3 surfaces without assuming the extendability of polarization, which is stronger than the results of Andr\'e and She. 
Moreover, as an application, we get the finiteness of twists of K3 surfaces via a finite extension of characteristic $0$ fields. 
\end{abstract}

\section{Introduction}
The Shafarevich conjecture for abelian varieties is a remarkable result which asserts the finiteness of isomorphism classes of abelian varieties of a fixed dimension over a fixed number field admitting good reduction away from a fixed finite set of finite places. This theorem was proved by Faltings (in the polarized case, see \cite{Faltings1983}) and Zarhin (in the unpolarized case, see \cite{Zarhin1985}).

In this paper, we shall prove an analogue of this theorem for K3 surfaces. For any discrete valuation field $K$ and a K3 surface $X$ over $K$, we say $X$ has good reduction if $X$ admits a smooth proper model over the valuation ring of $K$, as an algebraic space (see \cite[Section 1]{Liedtke2018}). 
\footnote{Note that it is natural to admit an integral model being an algebraic space rather than a scheme in the case of K3 surfaces (see \cite[Section 5.2]{Matsumoto2015a}).} 
Then one can formulate the analogue of the Shafarevich conjecture for K3 surfaces. 
Previously, this conjecture was studied by Andr$\e$ (\cite{Andre1996}) and She (\cite{She2017}) for polarized K3 surfaces. The goal of this paper is to generalize their results in terms of the unramifiedness of $\ell$-adic $\e$tale cohomology groups. Our main theorem is the following (for more generalized form, see Theorem \ref{main}).

\begin{theorem}[compare with Theorem \ref{main}]\label{introgenunp}
Let $F$ be a finitely generated field over $\Q$, and $R$ be a finite type algebra over $\Z$ which is a normal domain with the fraction field $F$. Then, the set
\[
\Shaf(F,R)\coloneqq
\left\{
X\left|
\begin{array}{l}
X \colon\textup{K3 surface over }F,\\

\textup{for any height $1$ prime ideal $\p \in \Spec R$,}\\
\textup{there exists a prime number $\ell \notin \p$}\\
\textup{such that $H^{2}_{\et}(X_{\overline{F}},\Q_{\ell})$ is unramified at $\p$}
\end{array}
\right.
\right\}/F\textup{-isom}
\]
is finite.
\end{theorem}

As a corollary, we have the original Shafarevich conjecture for K3 surfaces over finitely generated fields of characteristic $0$.

\begin{corollary}[Corollary \ref{shafunp}] \label{introshafunp}
Let $F$ be a finitely generated field over $\Q$, and $R$ be a finite type algebra over $\Z$ which is a normal domain with the fraction field $F$.
Then, the set
\[
\left\{X \left|
\begin{array}{l}
X\colon \textup{K3 surface over }F,\\
X \textup{ has good reduction at any height }1 \textup{ prime ideal }\p \in \Spec R
\end{array}
\right.\right\}/F\textup{-isom}
\]
is finite.
\end{corollary}

Note that our results are stronger than results of Andr$\e$ and She (see Remark \ref{introrem} for details).
Moreover, as an application of our cohomological generalization, we get the following corollary, which asserts the finiteness of twists of a K3 surface via a finite extension of characteristic $0$ fields. 

\begin{corollary}[Corollary \ref{fintwists}] \label{introfintwists}
Let $F$ be a field of characteristic $0$, $E/F$ be a finite extension, and $X$ be a K3 surface over $F$. Then, the set
\[
\Tw_{E/F}(X)\coloneqq
\{Y\colon \textup{K3 surface over }F \mid Y_{E}\simeq_{E}X_{E}
\}/F\textup{-isom}
\]
is finite. 
Here $Y_{E} \simeq_{E} X_{E}$ means the K3 surfaces $Y_{E}\coloneqq Y\otimes_{F}E$ and $X_{E}\coloneqq X\otimes_{F}E$ are isomorphic over $E$.
\end{corollary}

We note that our cohomological generalization is necessary for this application, i.e.\ the original statement of the Shafarevich conjecture (Corollary \ref{introshafunp}) is not enough to show Corollary \ref{introfintwists}.

Let us give some comments on the statement of Theorem \ref{introgenunp}. Theorem \ref{introgenunp} is motivated by the good reduction criterion for K3 surfaces given by Liedtke and Matsumoto (\cite{Liedtke2018}). For K3 surfaces over a Henselian discrete valuation field satisfying some assumptions, they showed the equivalence between the unramifiedness of $\ell$-adic $\et$ale cohomology groups and admitting good reduction after a finite unramified extension (\cite[Theorem 1.3]{Liedtke2018}). Note that the latter condition cannot be replaced by `admitting good reduction' (see \cite[Theorem 1.6]{Liedtke2018}), so our cohomological generalization is stronger than the original Shafarevich conjecture. Moreover, we deal with finitely generated fields of characteristic $0$ rather than number fields, motivated by the application to Corollary \ref{introfintwists}. 
In fact, Andr$\e$ also proved the Shafarevich conjecture for polarized K3 surfaces in this way (see \cite[Theorem 9.1.1]{Andre1996}, and see also the following Remark \ref{introrem}).

\begin{remark}\label{introrem}
Our results are stronger than previous results obtained by Andr$\e$ and She. 
To explain this, we briefly recall their results.
Andr$\e$ proved the Shafarevich conjecture for polarized K3 surfaces (\cite[Theorem 9.1.1]{Andre1996}), i.e.\ the finiteness of isomorphism classes of polarized K3 surfaces of fixed degree over a fixed number field  which admit good reduction away from a fixed finite set of finite places (actually, as stated above, Andr$\e$ dealt with finitely generated fields of characteristic $0$). 
Here, Andr$\e$ said that a polarized K3 surface $(X,L)$ admits good reduction if there exists a smooth proper model $\mathcal{X}$ of $X$ as a scheme such that the ample line bundle $L$ extends to an ample line bundle on $\mathcal{X}$.
Recently She proved it for polarized K3 surfaces of unspecified degree (\cite[Theorem 1.1.5]{She2017}). 
More correctly, She proved the finiteness of K3 surfaces over a fixed number field which admit good reduction as polarized K3 surfaces (without fixing polarization degree) away from a fixed finite set of finite places. 
Here, we remark that She's result does not cover K3 surfaces admitting a smooth proper model only as an algebraic space.
Moreover, there exists an example of a K3 surface admitting good reduction such that 
no smooth proper model has a polarization (therefore this K3 surface does not admit good reduction as polarized K3 surfaces) (see \cite[Section 5.2]{Matsumoto2015a}).
Therefore, Corollary \ref{introshafunp} is also stronger than previous results, even in the number field case.
\end{remark}

The strategy of the proof of Theorem \ref{introgenunp} is as follows. We basically take the approach of Andr$\e$ and She. 
We first show the polarized version of Theorem \ref{introgenunp} before dealing with the unpolarized case. 
To generalize the result obtained by Andr\'{e}, we should formulate the Kuga--Satake construction as preserving the finiteness. 
We achieve this by using the moduli interpretation of the Kuga--Satake construction introduced by Rizov (\cite{Rizov2010}).
In the perspective of the unpolarized case, we use the uniform Kuga--Satake construction introduced by She to study K3 surfaces of all degrees simultaneously.
Our proof is slightly different from She's proof, and here we will sketch the differences. 
In She's paper (\cite{She2017}), it is crucial to show that K3 surfaces admitting good reduction are sent to abelian varieties admitting good reduction via the uniform Kuga--Satake map. 
She proves this using integral canonical models of certain Shimura varieties (the argument like `$\oo$-valued points go to $\oo$-valued points'). However, in our case, we do not assume that each K3 surface admits a smooth proper model, so instead of She's method, we use the N$\e$ron--Ogg--Shafarevich criterion for abelian varieties (this approach is already known by Andr\'{e} \cite{Andre1996} see also \cite{Imai2020}). 
For this purpose, we study She's uniform Kuga--Satake construction in detail in Section $3$. 
Note that our proof does not require the theory of integral canonical models of Shimura varieties. 
To simplify the arguments, we need a section of the natural map from the $\GSpin$ Shimura variety to the $\SO$ Shimura variety (see \ref{subsectionuks}). To get such a section, we should work with a level structure that has a sufficiently small $\Z_{2}$-component (see Remark \ref{rem2adic}). 
So we should suppose the unramifiedness of $2$-adic representation to overcome that general K3 surfaces may not admit level structure (see Proposition \ref{redp}).
However, this assumption is not essential since $\ell$-independence of the unramifiedness is true in a general situation (see Lemma \ref{lindep}).
Note that Lemma \ref{lindep} is essentially known by Madapusi Pera, Matsumoto \cite{Matsumoto2016} and Imai--Mieda \cite{Imai2020} (see Remark \ref{remlindep}).

The outline of this paper is as follows. In Section $2$, we will recall the basic results on K3 surfaces, and define the moduli space of K3 surfaces introduced by Rizov and Madapusi Pera. 
In Section $3$, we will define several algebraic groups to introduce the uniform Kuga--Satake abelian varieties, and study their basic properties. 
In Section $4$, we will prove the main theorem in a little weaker form (i.e.\ only considering $2$-adic cohomology) by using the results of Section $3$ and the arguments given by Andr$\e$ and She.
In Section $5$, we will see an $\ell$-independence of the unramifiedness by using Matsumoto's result on weight filtrations \cite[Theorem 3.3]{Matsumoto2016}. We also prove a crystalline analogue of it by using Ochiai's $\ell$-independence results \cite{Ochiai1999} and the Kuga--Satake abelian varieties (as above, this result is essentially proved in \cite{Imai2020}). 
In Section $6$, we will complete the proof of the main theorem combining the results in Section $4$ and Section $5$, and prove the finiteness of twists via a finite extension of characteristic $0$ fields.

\subsection*{Acknowledgments}
The author is deeply grateful to my advisor Naoki Imai for his deep encouragement and helpful advice. He carefully read the draft version of this paper and pointed out a lot of mistakes and typos. The author is greatly indebted to Tetsushi Ito for his warm encouragement, many comments, and invaluable suggestions for the proof. Moreover, the author would like to thank Kazuhiro Ito for helpful comments about the definition of the moduli of K3 surfaces, and Yoichi Mieda for valuable comments about Section $5$. 
The author also thanks the referee for many helpful comments, including suggestions on Remark \ref{remunifks} and Proposition \ref{good}.
The author is supported by the FMSP program at the University of Tokyo.

\section{K3 surfaces and their moduli}
\subsection{Basic definitions for K3 surfaces}
In this subsection, we give definitions and basic notations about K3 surfaces. 
\begin{definition}\label{K3} \textup{ }
\begin{enumerate}
\item
For any field $k$, a \emph{K3 surface over k} is a smooth proper surface $X$ over $k$ with $\Omega^{2}_{X/k} \simeq \mathcal{O}_{X}$ and $H^{1}(X, \mathcal{O}_{X})=0.$
\item
For any scheme $S$, a \emph{K3 family over $S$} is a smooth proper algebraic space $\mathcal{X}$ over $S$ whose geometric fibers are K3 surfaces.

\end{enumerate}
\end{definition}
\begin{remark}
For any field $k$, a K3 family over $k$ is automatically a K3 surface over $k$ since smooth proper algebraic spaces of dimension $2$ over a field are schemes.
\end{remark}

\begin{definition}[{\cite[Definition 3.2.2, Definition 3.2.3]{Rizov2006}}] \textup{ }
\begin{enumerate}
\item
A \emph{polarization} on a K3 family $\pi\colon \mathcal{X}\rightarrow S$ is an element $\lambda \in \Pic_{\mathcal{X}/S}(S)$ whose pullback by any geometric point of $S$ is an ample line bundle. Here $\Pic_{\mathcal{X}/S}$ is the relative Picard functor.
\item
A polarization $\lambda$ is \emph{primitive} if its pullback by any geometric point of $S$ is primitive, i.e.\ not divisible by an integer greater than $1$.
\item
A polarization $\lambda$ is \emph{of degree $2d$} if its pullback by any geometric point of $S$ has degree $2d$, i.e.\ its self intersection number is $2d$.

\end{enumerate}
\end{definition}

\begin{remark}
Let $F$ be a subfield of $\C$. For a K3 surface $X$ over $F$, the relative Picard functor $\Pic_{X/F}$ is represented by a scheme, thus $\Pic_{X/F}(F)= \Pic(X_{\overline{F}})^{\Gal(\overline{F}/F)}$ is a primitive sublattice in $H^{2}(X(\C), \Z(1))$ via the Chern class map as in \cite[Lemma 2.2.3]{She2017}. Hence there exists a primitive polarization for each $X$ (by dividing a polarization by an integer greater than $1$ if necessary). Note that the inclusion $\Pic(X) \subset \Pic_{X/F}(F)$ may be proper in general, though it always has a finite cokernel (see \cite[Chapter 17, Section 2.2]{Huybrechts2016}).
\end{remark}

\begin{definition}\label{lattice}
\begin{enumerate}
\item 
A \emph{K3 lattice $\la_{K3}$} is a unimodular lattice\footnote{In this paper, a ($\Z$-)lattice means a finite free $\Z$-module with a symmetric bilinear pairing valued in $\Z$, and $\langle c \rangle$ means the $\Z$-lattice of rank 1 given by $(a,b)= cab$.} of signature $(19,3)$ which is defined as
\begin{align*}
\la_{K3}\coloneqq \mathbb{E}_{8}^{\oplus2} \oplus \mathbb{H}^{\oplus3},
\end{align*}
where $\mathbb{E}_{8}$ is the (positive signature) $E_{8}$-lattice as in \cite[Chapter 14, Example 0.3]{Huybrechts2016}, and $\mathbb{H}$ is the hyperbolic plane. 

\item
Consider the last component $\mathbb{H}\subset \la_{K3}$, and take $e,f \in \mathbb{H} \subset \la_{K3}$ satisfying 
\[
(e,f)=(f,e)=1,\quad (e,e)=(f,f)=0.
\]
Let $v_{d}\coloneqq e-df$. Then the \emph{degree $2d$ primitive part of $\la_{K3}$} is defined as
\[
\la_{d}\coloneqq v_{d}^{\perp} \simeq \mathbb{E}_8^{2} \oplus \mathbb{H}^{2} \oplus \langle 2d\rangle.
\] 
The lattice $\la_{d}$ is a primitive sublattice of $\la_{K3}$, and $\disc(\la_{d})=2d$.
\end{enumerate}
\end{definition}

\begin{remark}[{\cite[Remark 2.3.2]{Rizov2006}}]\label{remlattice}
For a K3 surface $X$ over $\C$ and its primitive polarization $L$ of degree $2d$,
there exists an isomorphism 
\begin{align*}
    (H^{2}(X(\C),\Z(1)),-\cup) \simeq \la_{K3}
\end{align*}
which sends $\ch_{\Z}(L)$ to $v_{d}$. 
Here $-\cup$ denotes the minus of the cup product.
Therefore, for a primitively polarized K3 surface $(X,L)$ of degree $2d$ over a field $F$ which is contained in $\C$, we sometimes identify $H^{2}(X(\C),\Z(1))$ with $\la_{K3}$, $H^{2}_{\et}(X_{\overline{F}},\widehat{\Z}(1))$ with $\la_{K3,{\widehat{\Z}}}\coloneqq \la_{K3} \otimes_{\Z}\widehat{\Z}$,
$P^{2}((X(\C),L_{\C}),\Z(1))$ with $\la_{d}$,
and $P^{2}_{\et}((X_{\overline{F}},L_{\overline{F}}),\widehat{\Z}(1))$ with $\la_{d,\widehat{\Z}}\coloneqq \la_{d}\otimes_{\Z}\widehat{\Z}$. 
Here, we denote the primitive parts of the singular cohomology group and the $\et$ale cohomology group by
\begin{align*}
P^{2}((X(\C),L_{\C}),\Z(1)) &\coloneqq \ch (L_{\C})^{\perp} \subset H^{2}(X(\C),\Z),\\
P^{2}_{\et}((X_{\overline{F}},L_{\overline{F}}),\widehat{\Z}(1)) &\coloneqq \ch_{\widehat{\Z}}(L)^{\perp}\subset H^{2}_{\et}(X_{\overline{F}},\widehat{\Z}(1)).
\end{align*}
To simplify the notation, in the following of this paper, we omit the pairing. We denote $(H^2(X(\C),\Z(1)), -\cup)$ by $H^{2}(X(\C),\Z(1))$, and same with others.

\end{remark}

\begin{definition}\label{discriminantkernel}
The \emph{discriminant kernel of $\la_{d}$} is 
\[
D_{d} \coloneqq \{g\in \SO(\la_{d, \widehat{\Z}}) \mid g \textup{ acts trivially on } \la_{d,\widehat{\Z}}^{\vee}/\la_{d,\widehat{\Z}} \} .
\]
Note that $D_{d}$ is a compact open subgroup of $\SO(\la_{d,\A_{f}}).$ For any prime number $\ell$, we denote its $\Z_{\ell}$-component by $(D_{d})_{\ell}$.
\end{definition}

\begin{prop}[{\cite[Lemma 2.6]{MadapusiPera2016}}]\label{disc} 
There is a natural identification
\[
D_{d} =
\{\tilde{g}\in \SO(\la_{K3, \widehat{\Z}}) \mid \tilde{g}(v_{d})= v_{d} \}.
\]
\end{prop}
\begin{proof}
This is proved in \cite[Lemma 2.6]{MadapusiPera2016}. We include its proof because we need to recall the identification explicitly. 
Let $\ell$ be any prime number, and we will verify this claim for each $\Z_{\ell}$-component.
First, we will define a map from the left-hand side to the right-hand side. 
For 
\[
g_{\ell} \in (D_{d})_{\ell}= \{g\in \SO(\la_{d,\Z_{\ell}})\mid g \textup{ acts trivially on }\la_{d,\Z_{\ell}}^{\vee}/\la_{d,\Z_{\ell}}\},
\]
define $\tilde{g}_{\ell}$ as the image of $g_{\ell}$ via the composition of the following
\[
\SO(\la_{d, \Z_{\ell}})\hookrightarrow \SO(\la_{d, \Q_{\ell}}) \hookrightarrow \SO(\la_{K3, \Q_{\ell}}).
\]
Then we have $\tilde{g}_{\ell} v_{d}= v_{d}$. 
We will show that $\tilde{g}_{\ell} \la_{K3,\Z_{\ell}}= \la_{K3,\Z_{\ell}}$. Consider the morphisms 
\begin{align}
\la_{K3,\Z_{\ell}} \simeq \la_{K3,\Z_{\ell}}^{\vee}
\hookrightarrow \la_{d,\Z_{\ell}}^{\vee} \oplus \langle v_{d} \rangle^{\vee} \twoheadrightarrow \la_{d, \Z_{\ell}}^{\vee}.
\end{align}
For any $v\in \la_{K3,\Z_{\ell}}$, denote its image in $\la_{d,\Z_{\ell}}^{\vee} \oplus \langle v_{d}\rangle^{\vee}$ by $u_{1}+ u_{2}$. Then we have
\[
\tilde{g}_{\ell}(u_{1}+u_{2})=g_{\ell}(u_{1})+u_{2}=(g_{\ell}(u_{1})-u_{1}) + (u_{1}+u_{2}) \in \la_{K3, \Z_{\ell}},
\]
because $g_{\ell}$ acts trivially on $\la_{d,\Z_{\ell}}^{\vee}/\la_{d,\Z_{\ell}}$. Hence $\tilde{g}_{\ell}^{\pm1}\la_{K3,\Z_{\ell}}\subset \la_{K3,\Z_{\ell}}$, thus $\tilde{g}_{\ell} \la_{K3, \Z_{\ell}}=\la_{K3, \Z_{\ell}}$, and we can define the desired map.

Next, we will define a map from the right-hand side to the left-hand side. For $\tilde{h}_{\ell}\in \SO(\la_{K3,\Z_{\ell}})$ such that $\tilde{h}_{\ell}v_{d}=v_{d}$, we can associate $h_{\ell}\in \SO(\la_{d, \Z_{\ell}})$ as the restriction of $\tilde{h}_{\ell}$. We can show that $h_{\ell}$ acts trivially on $\la_{d, \Z_{\ell}}^{\vee}/\la_{d, \Z_{\ell}}$. Indeed, because the embedding $\la_{d}\hookrightarrow \la_{K3}$ is primitive, the composition of $(1)$ is surjective, so for any $u_{1} \in \la_{d,\Z_{\ell}}^{\vee},$ there exists $u_{2} \in \langle v_{d} \rangle^{\vee}$ such that $u_{1}+u_{2}\in \la_{K3,\Z_{\ell}}.$ 
Thus we have
\[
h_{\ell}(u_{1})-u_{1}= \tilde{h}_{\ell}(u_{1}+u_{2})-(u_{1}+u_{2})\in \la_{K3,\Z_{\ell}}\cap \la_{d,\Q_{\ell}}=\la_{d,\Z_{\ell}}.
\]

Clearly, the above maps are inverses of each other, so it finishes the proof.
\end{proof}

\subsection{Moduli spaces of K3 surfaces and the Torelli theorem}
In this subsection, we recall the Torelli theorem in terms of moduli spaces.
First, we recall the definition of the moduli space of K3 surfaces with oriented level structures. See \cite[Section 6]{Rizov2006}, \cite[Section 3]{MadapusiPera2015}, \cite[Section 5]{Ito2018} for details.

We define a groupoid-valued moduli functor $M_{2d,\Q}^{\circ}$ by
\[
M_{2d,\Q}^{\circ}(S) \coloneqq
\left\{ 
(\pi\colon \mathcal{X} \rightarrow S, \lambda \in \Pic_{X/S}(S)) 
\left| 
\begin{array}{l}
\textup{$\pi \colon$K3 family over $S$,}\\
\textup{$\lambda \colon$primitive polarization of degree $2d$}
\end{array}
\right.\right\}
\]
for any $\Q$-scheme $S$.
Let $\widetilde{M}_{2d,\Q}^{\circ}$ be the twofold finite $\e$tale cover constructed by Madapusi Pera (\cite[Section 5]{MadapusiPera2015}) which parameterizes orientations.
Then, for any $S \rightarrow \widetilde{M}_{2d,\Q}^{\circ}$ we get $(\pi, \lambda, \nu)$ where $(\pi, \lambda)$ is as above, and $\nu$ is an isometry of $\widehat{\Z}$-local systems 
\[
\nu \colon
\underline{\det\la_{d, \widehat{\Z}}} \simeq \det P^{2}\pi_{\ast}\widehat{\Z}
\]
such that for any $s\in S(\C)$, the isometry $\nu$ restricts to an isometry
\begin{align*}
\nu_{s} \colon \det \la_{d} \simeq \det P^{2}(\mathcal{X}_{s}, \Z).
\end{align*}
Here, we put
\[
P^2\pi_{\ast}\widehat{\Z}(1) \coloneqq \ch_{\widehat{\Z}}(\lambda)^{\perp} \subset R^2\pi_{\ast}\widehat{\Z}(1),
\] 
where $\ch_{\widehat{\Z}}(\lambda)$ is the Chern class of $\lambda$ (\cite[3.10]{MadapusiPera2015}). 
Let $\K\subset D_{d}$ be a compact open subgroup. 
For any scheme $S$ over $\widetilde{M}_{2d,\Q}^{\circ}$, 
one can define the $\e$tale sheaf $I$ by 
\[
I(T) \coloneqq
\left\{ g\colon \underline{\la_{K3, \widehat{\Z}}} \rightarrow R^{2}\pi|_{T \ast}\widehat{\Z}(1) \left|
\begin{array}{l}
g\colon \textup{isometry, }\\
g(v_{d})= \ch_{\widehat{\Z}}(\lambda),\\
\textup{$\det g$ induces $\nu|_{T}$}
\end{array}
\right.\right\},
\]
for any $\e$tale morphism $T \rightarrow S$. 
A $\K$-level structure on $S \rightarrow \widetilde{M}_{2d,\Q}^{\circ}$ is a section $\alpha \in H^{0}(S, I/\K)$, where $\K$ acts on $I$ through $\la_{K3,\widehat{\Z}}$. Then, one can define the moduli functor $M_{2d,\K,\Q}^{\circ}$ over $\widetilde{M}_{2d,\Q}^{\circ}$ which parameterizes $\K$-level structures. 
For simplicity, we write an each element of $M_{2d,\K,\Q}^{\circ}(S)$ as $(\mathcal{X}, \lambda, \nu, \alpha)$.
Moreover, for any field $F$ of characteristic $0$, we denote the base change by $M_{2d,\K,F}^{\circ}$.

\begin{definition}\phantomsection\label{SO}
\begin{enumerate}
\item
$\SO_{\la_{d}}$ is an algebraic group over $\Q$ whose $R$-valued points are given by
\[
\SO_{\la_{d}}(R)\coloneqq \{ g\in \SL(\la_{d,R})\mid (gv, gw)=(v, w), \textup{ for any } v, w \in \la_{d,R}\}.
\]

\item
We put
\[
\Omega^{\pm}_{\SO_{\la_{d}}}\coloneqq 
\{
\textup{oriented negative definite planes in $\la_{d,\mathbb{R}}$}
\}.
\]
Then $\Omega^{\pm}_{\SO_{\la_{d}}}$ is naturally identified with $X_{\SO_{\la_{d}}}$ which gives the Shimura datum $(\SO_{\la_{d}}, X_{\SO_{\la_{d}}})$ with a reflex field $\Q$.
Actually, $X_{\SO_{\la_{d}}}$ is isomorphic to $X_{\GSpin_{\la_{d}}}$ which is defined as in Definition \ref{shimurala} via the adjoint representation.
\end{enumerate}
\end{definition}

Here, we quickly state the moduli interpretation of the Torelli theorem over $\Q$.

\begin{prop}[{The Torelli theorem,
\cite[Corollary 5.4, Theorem 5.8]{MadapusiPera2015}}]\label{Torelli}
Let $\K\subset D_{d}$ be a compact open subgroup. 
Moreover, assume that $\K$ is contained in the principal level $n$ congruence subgroup of $\SO(\la_{d,\widehat{\Z}})$ with $n \geq 3$.
Then $M_{2d,\K,\Q}^{\circ}$ is representable by a scheme,
and moreover there is the period map which is an $\e$tale morphism between $\Q$-schemes
\[
j\colon M_{2d,\K,\Q}^{\circ}  \rightarrow \Sh_{\mathbb{K}}(\SO_{\la_{d}}, X_{\SO_{\la_{d}}}).
\]
Here $\Sh_{\mathbb{K}}(\SO_{\la_{d}}, X_{\SO_{\la_{d}}})$ is the canonical model of the Shimura variety over $\Q$.
\end{prop}

In Proposition \ref{uks}, we will use the more detailed properties of the period map $j$.

\section{The uniform Kuga--Satake construction}
In this section, we recall the definition and properties of the Kuga--Satake construction. 
In this section, we use only the uniform Kuga--Satake construction introduced by She. 
In fact, the classical Kuga--Satake construction is enough for proving the polarized case (Theorem \ref{genp}), but we need She's methods to prove the unpolarized case (Theorem \ref{genunp}). Hence we omit the classical Kuga--Satake construction for avoiding some repetitions.

\subsection{Preparation I}
In this and next subsection, we will define several algebraic groups and their adelic subgroups which play an important role in the Kuga--Satake construction. In this subsection, we discuss objects related with the lattice $\la_{d}.$

For any algebra $R$ and any quadratic space $\mathcal{N}$ over $R$, we denote the Clifford algebra (resp.\ even Clifford algebra) of $\mathcal{N}$ by $C(\mathcal{N})$ (resp.\ $C^{+}(\mathcal{N})$). 
\begin{definition}\label{shimura}
$\GSpin_{\la_{d}}$ is an algebraic group over $\Q$, whose $R$-valued points are given by
\[
\GSpin_{\la_{d}}(R) \coloneqq \{z\in C^{+}(\la_{d, R})^{\times} \mid z\la_{d, R}z^{-1}= \la_{d, R}\}.
\]

\end{definition}

\begin{remark}\phantomsection\label{morp}
\begin{enumerate}
\item
There exists the following natural homomorphism of algebraic groups over $\Q$
\[
f_{d} \colon \GSpin_{\la_{d}} \rightarrow \SO_{\la_{d}} ; g \mapsto (l \mapsto glg^{-1}).
\]

\item 
For any $\Z$-algebra $R$, we put 
\[
\GSpin(\la_{d,R})\coloneqq \{z \in C^{+}(\la_{d,R})^{\times}\mid z\la_{d,R}z^{-1}=\la_{d,R}\}.
\]
Then, for any prime number $\ell$, we can define 
\[
f_{d}\colon \GSpin(\la_{d,\Z_{\ell}})\rightarrow \SO(\la_{d,\Z_{\ell}})
\] 
by the conjugation.
Moreover, it is easy to confirm the folowing identity
\[
\GSpin(\la_{d,\Z_{\ell}})=\GSpin(\la_{d,\Q_{\ell}})\cap C^{+}(\la_{d, \Z_{\ell}})^{\times}.
\]

\item
For any $\Z$-algebra $R$, we will use the notation $\GSpin(\la_{K3,R})$ in the similar sense as in (2). Moreover, for any prime number $\ell$, we denote the conjugation map $\GSpin(\la_{K3,\Z_{\ell}})\rightarrow \SO(\la_{K3,\Z_{\ell}})$ by $f_{K3}$. As in (2), it follows that 
\[
\GSpin(\la_{K3,\Z_{\ell}})= \GSpin(\la_{K3,\Q_{\ell}})\cap C^{+}(\la_{K3,\Z_{\ell}})^{\times}.
\]
\end{enumerate}
\end{remark} 

\begin{lemma}[{\cite[(2.6.1)]{MadapusiPera2016}}]\label{cliffordlemma}
Let $\ell$ be any prime number.
Through the natural inclusion $C^{+}(\la_{d,\Z_{\ell}})\subset C^{+}(\la_{K3,\Z_{\ell}})$, we have
\[
C^{+}(\la_{d,\Z_{\ell}})=\{z\in C^{+}(\la_{K3,\Z_{\ell}})\mid v_{d}z=z v_{d} \}.
\] 
Moreover, the above inclusion induces an embedding
\[
\GSpin(\la_{d,\Z_{\ell}})\subset \GSpin(\la_{K3,\Z_{\ell}}).
\]
\end{lemma}
\begin{proof}
The first claim is essentially proved in \cite[(2.6.1)]{MadapusiPera2016}. For the sake of completeness, we recall the proof.
For the first claim, both sides of the desired identity are primitive $\Z_{\ell}$-modules in $C^{+}(\la_{K3,\Z_{\ell}})$. Thus, it is enough to show that
\[
C^{+}(\la_{d,\Q_{\ell}})=\{z \in C^{+}(\la_{K3,\Q_{\ell}})\mid v_{d}z=z v_{d}\}.
\]
It can be easily verified by using a basis of $\la_{K3,\Q_{\ell}}$ which is given by a basis of $\la_{d,\Q_{\ell}}$ and $v_{d}$.
For the second claim, by Remark \ref{morp} (2) and (3), we can reduce the problem to the obvious inclusion
$\GSpin(\la_{d,\Q_{\ell}})\subset \GSpin(\la_{K3,\Q_{\ell}}).$ 
\end{proof}

\begin{definition}[{\cite[Section 4.4]{Andre1996}, \cite[Example 5.1.4]{Rizov2006}}] 
For any positive integer $n$, we define a compact open subgroup $\K_{d,n}^{\spp} \subset \GSpin_{\la_{d}}(\A_{f})$ 
by
\[
\K_{d,n}^{\spp} \coloneqq \{g \in \GSpin(\la_{d,\widehat{\Z}}) \mid
g=1 \textup{ in } C^{+}(\la_{d,\widehat{\Z}/n\widehat{\Z}})\}.
\]
\end{definition}

\begin{prop}[{cf.\ \cite[Section 4.4]{Andre1996}, \cite[Section 4.4]{MadapusiPera2015}}]\label{openness}
\[
D_{d}(n) \coloneqq f_{d}(\mathbb{K}_{d,n}^{\spp})\subset \SO(\la_{d,\widehat{\Z}}) 
\]
is a compact open subgroup of $D_{d}$.
\end{prop}

\begin{proof}
First, we shall show that $D_{d}(n)$ is contained in $D_{d}$. Lemma \ref{cliffordlemma} shows that
\[
f_{K3}(\GSpin(\la_{d,\Z_{\ell}}))\subset \{g\in \SO(\la_{K3,\Z_{\ell}})\mid gv_{d}=v_{d}\},
\]
thus the desired inclusion follows from Proposition \ref{disc}.

For the openness, it is enough to show that for any
$\ell$ not dividing $2dn$, the $\Z_{\ell}$\textup{-component of }$D_{d}(n)$ is equal to $\SO(\la_{d,\Z_{\ell}})$.
It follows from \cite[Section 4.4]{Andre1996}.
\end{proof}

The following proposition gives more information about $D_{d}(n).$
\begin{prop}\label{surjtodisc}
For any odd prime number $\ell \neq 2$, we have
\[
f_{d}(\GSpin(\la_{d,\Z_{\ell}}))=(D_{d})_{\ell}.
\]
If $\ell=2$, as a subset of $\SO(\la_{K3,\Z_{2}})$, we have
\[
f_{d}(\GSpin(\la_{d,\Z_{2}}))=(D_{d})_{2}\cap f_{K3}(\GSpin(\la_{K3,\Z_{2}})).
\]
\end{prop}

\begin{proof}
If $\ell$  does not divide $2d$, these results are essentially shown in the proof of Proposition \ref{openness}. 

First, for any prime number $\ell$, we have $f_{d}(\GSpin(\la_{d,\Z_{\ell}}))\subset (D_{d})_{\ell}$ as in the proof of Proposition \ref{openness}.
We assume $\ell \neq 2$. 
For any $g\in (D_{d})_{\ell}\subset \SO(\la_{K3,\Z_{\ell}}),$ by the same argument as in
\cite[Section 4.4]{Andre1996} (here we use $\ell\neq 2$),
there exists $z \in \GSpin(\la_{K3,\Z_{\ell}})$ such that $f_{K3}(z)=g$. 
Proposition \ref{disc} implies $zv_{d}z^{-1}=v_{d}$, and so in fact, $z\in C^{+}(\la_{d,\Z_{\ell}})^{\times}$ by Lemma \ref{cliffordlemma}.  By Proposition \ref{disc}, $z$ stabilizes $\la_{d,\Z_{\ell}}$ via conjugation,
thus $z\in \GSpin(\la_{d,\Z_{\ell}})$ and it finishes the proof of the first claim.
If $\ell=2,$ the second claim follows by the same arguments.
\end{proof}
\begin{remark}\label{rem2adic}
Unfortunately, if $\ell =2,$ we have $f_{d}(\GSpin(\la_{d,\Z_{2}})) \neq (D_{d})_{2}$. Indeed, there exists $g_{2}\in(D_{d})_{2}$ which is non-trivial in $\SO(\la_{d,\Z/2\Z})$ (for example, permutation of two components $\mathbb{H}_{\Z_{2}} \subset \la_{d,\Z_{2}}$), though any element in the image of $f_{d}$ is trivial there.
\end{remark}

\begin{corollary}\label{D_{d}(n)}
Let $(D_{d}(n))_{\ell}$ be the $\Z_{\ell}$-component of $D_{d}(n),$ and $n_{\ell}$ be the $\ell$-part of $n$. 
Then, for any prime number $\ell\neq 2$, we have 
\[
[(D_{d})_{\ell}\colon(D_{d}(n))_{\ell}]\leq n_{\ell}^{(2^{20})}.
\]

Moreover, there exists a positive integer $N$ which is independent of $d$ and $n$ 
such that 
\[
[(D_{d})_{2}\colon(D_{d}(n))_{2}] \leq N\cdot n_{2}^{(2^{20})}.
\]
\end{corollary}

\begin{proof}
Assume $\ell\neq 2$.
We have the following commutative diagram.
\[
\xymatrix{
\GSpin(\la_{d,\Z_{\ell}}) \ar@{->>}[r] & (D_{d})_{\ell}\\
(\mathbb{K}_{d,n}^{\spp})_{\ell} \ar@{^{(}->}[u] \ar@{->>}[r] & (D_{d}(n))_{\ell} \ar@{^{(}->}[u]
}
\]
Here we have
\[
(\mathbb{K}_{d,n}^{\spp})_{\ell}=\{g\in \GSpin(\la_{d,\Z_{\ell}})\mid g=1 \textup{ in } C^{+}(\la_{d,\Z_{\ell}/n_{\ell}\Z_{\ell}})\}.
\]
Since
\[
\#(C^{+}(\la_{d,\Z_{\ell}/n_{\ell}\Z_{\ell}})^{\times})\leq \#(C^{+}(\la_{d,\Z_{\ell}/n_{\ell}\Z_{\ell}}))= n_{\ell}^{(2^{20})},
\]
the index of $(\mathbb{K}_{d,n}^{\spp})_{\ell}$ in $\GSpin(\la_{d,\Z_{\ell}})$ is bounded by $n_{\ell}^{(2^{20})}$, and it finishes the proof of the first claim.

For the second claim, we put 
\[
N\coloneqq [\SO(\la_{K3,\Z_{2}}):f_{K3}(\GSpin(\la_{K3,\Z_{2}}))].
\]
Then, by the second claim of Proposition \ref{surjtodisc} and the above arguments, we have
\begin{align*}
[(D_{d})_{2}\colon D_{d}(n)_{2}]&\leq
[(D_{d})_{2}\colon (D_{d})_{2}\cap f_{K3}(\GSpin(\la_{K3,\Z_{2}}))]\cdot n_{2}^{(2^{20})} \\
& \leq N\cdot n_{2}^{(2^{20})}.
\end{align*}
\end{proof}

\subsection{Preparation II}
Here, we will introduce a even unimodular lattice $\la$ of signature $(26,2)$ which contains all $\la_{d}$. Then we will define related objects as in the previous subsection.

\begin{prop}[{see also \cite[Lemma 3.3.1]{She2017}}]
We put 
\[
\mathcal{L} \coloneqq \mathbb{E}_{8}^{3} \oplus \mathbb{H}^{2}.
\]
For any positive integer $d$, there exists a primitive embedding of lattices
\[
i_{d}\colon \mathcal{L}_{d}\hookrightarrow \mathcal{L}.
\]
\end{prop}

\begin{proof}
See \cite[Corollary 1.12.3]{Nikulin1979}.
\end{proof}

\begin{remark}\phantomsection\label{remunifks}
\begin{enumerate}
\item
Since $\la$ is unimodular, the group $\SO(\la_{\widehat{\Z}})$ is the discriminant kernel of $\la$, which is defined as in Definition \ref{discriminantkernel}.
\item
To get a primitive embedding into a self-dual lattice, the lattice $\mathbb{E}_{8}^2 \oplus \mathbb{H}^{2} \oplus \langle1\rangle^{5}$ is enough (see \cite[Lemma 3.3.1]{She2017}). 
However, we require that $\la$ comes from a quadratic space for using the usual definition of Clifford algebras.
\end{enumerate}
\end{remark}

Next, we will define related algebraic groups and Shimura data for $\la$ as in Definition \ref{SO} and Definition \ref{shimura}.

\begin{definition}\phantomsection\label{shimurala}
\begin{enumerate}
\item
$\GSpin_{\la}$ is the algebraic group over $\Q$ whose $R$-valued points are given by
\[
\GSpin_{\la}(R)\coloneqq
\{z\in C^{+}(\la_{R})^{\times} \mid z \la_{R}z^{-1}=\la_{R}\}.
\]
\item
Take a 2-dimensional negative definite subspace of $\la_{\Q}$, and let $e_{1}, e_{2}$ be its orthogonal basis. Let $e'_{1}, e'_{2}$ be an orthonormal basis over $\R$ which are given by constant multiples of $e_{1}, e_{2}$, and $J\coloneqq e'_{1}e'_{2}\in C^{+}(\la_{\R})$. 
Let $\psi$ be the following map
\[
\psi \colon \dt \rightarrow \GSpin_{\la,\R} ; \alpha+\beta i \mapsto \alpha+\beta J,
\] 
and $X_{\GSpin_{\la}}$ be a $\GSpin_{\la}(\R)$-conjugacy class containing $\psi$. 

\item
$\SO_{\la}$ is the algebraic group over $\Q$ whose $R$-valued points are given by
\[
\SO_{\la}(R)\coloneqq
\{g\in \SL(\la_{R}) \mid (gv, gw)= (v,w) \textup{ for any } v,w \in \la_{R}\}.
\]

\item
$X_{\SO_{\la}}$ is the (isomorphic) image of $X_{\GSpin_{\la}}$ via the adjoint representation $\GSpin_{\la} \rightarrow \SO_{\la}$.

\item
For $V\coloneqq C(\la)$ and a fixed $a\in V$ which is a constant multiple of $e_{1}e_{2}$, define $\phi_{a}\colon V\times V \rightarrow \Z$ as $\phi_{a}(x, y)\coloneqq \mathrm{tr}_{V/\Q}(xay^{\ast}).$ Here $\mathrm{tr}_{V/\Q}(x)$ means the trace of a left multiplication map by $x$ as in \cite[Chapter 4, Section 2.2]{Huybrechts2016}, and $\ast$ denotes the natural anti-automorphism on the Clifford algebra. 
Then $\phi_{a}$ is a non-degenerate alternative form. We denote its degree by $r$.
Let $\GSp_{V,a}$ be the algebraic group over $\Q$ whose $R$-valued points are given by 
\[
\GSp_{V, a}(R) \coloneqq 
\left\{g \in \GL(V_{R}) \left| 
\begin{array}{l}
\textup{there exists $c\in R^{\times}$ such that }\\
\phi_{a}(gx, gy) = c\phi_{a}(x, y) \textup{ for any $x, y \in V_{R}$}
\end{array}
\right.\right\}.
\]
Let $(\GSp_{V,a}, X_{\GSp_{V,a}})$ be the Shimura datum associated with $(V, \phi_{a})$. 
\end{enumerate}
\end{definition}

\begin{remark}
\begin{enumerate}
\item
As in Remark \ref{morp} (1), we can define a homomorphism
\[
f \colon \GSpin_{\la} \rightarrow \SO_{\la} ; g \mapsto (l \mapsto glg^{-1}).
\]
Moreover,  it induces a morphism of Shimura data
\[
(\GSpin_{\la},X_{\GSpin_{\la}})\rightarrow (\SO_{\la},X_{\SO_{\la}}).
\]
\item
We can define a homomorphisms
\[
h \colon \GSpin_{\la} \rightarrow \GSp_{V,a} ; g \mapsto (v \mapsto gv).
\]
Moreover, it induces an embedding of Shimura data 
\[
(\GSpin_{\la},X_{\GSpin_{\la}})\rightarrow (\GSpin_{V,a}, X_{\GSpin_{V,a}})
\]
by our definition of $a$ (see \cite[Chapter 4, Section 2.2]{Huybrechts2016}).

\item
We will use a similar notation as in Remark \ref{morp} (2), (3) for $\la$.
\end{enumerate}
\end{remark}

\begin{definition}
For any positive integer $n,$ we define compact open subgroups $\K_{n}^{\spp} \subset \GSpin_{\la}(\A_{f})$ and $\K_{n} \subset \GSp_{V,a}(\A_{f})$ by
\begin{align*}
\K_{n}^{\spp} &\coloneqq \{g \in \GSpin(\la_{\widehat{\Z}}) \mid
g=1 \textup{ in } C^{+}(\la_{\widehat{\Z}/n\widehat{\Z}})\},\\
\K_{n} &\coloneqq \{g \in \GSp_{V,a}(\A_{f}) \mid
g V_{\widehat{Z}}= V_{\widehat{\Z}},
\, g \textup{ acts trivial on } V_{\widehat{\Z}/n \widehat{\Z}}\}.
\end{align*}
\end{definition}

\begin{remark}\phantomsection\label{univab}
\begin{enumerate}
\item
One can show that $h(\K_{n}^{\spp})\subset\K_{n}$ and $h^{-1}(\K_{n})=\K_{n}^{\spp}$.
Moreover, our definition of $\K_{n}$ coincides with $\Lambda_{n}$ in Rizov's paper \cite[Section 5.5]{Rizov2010}. 
Therefore, as in \cite[Section 5.5]{Rizov2010}, we have an embedding
\[
\Sh_{\K_{n}}(\GSp_{V,a}, X_{\GSp_{V,a}}) \hookrightarrow \mathcal{A}_{g, \sqrt{r}, n, \Q}.
\]
Here, we put $g\coloneqq 2^{27}$, and $\mathcal{A}_{g, \sqrt{r}, n, \Q}$ is the moduli space of $g$-dimensional degree $r$ polarized abelian schemes with level $n$-structure. 

\item
The lattice embedding $i_{d}\colon \la_{d} \hookrightarrow \la$ indeces a morphism of algebraic groups $i_{d}\colon\SO_{\la_{d}}\rightarrow \SO_{\la}$. It induces an embedding of Shimura data
\[
(\SO_{\la_{d}},X_{\SO_{\la_{d}}})\rightarrow(\SO_{\la},X_{\SO_{\la}}).
\]
\item
One can show that $D(n)\coloneqq f(\K_{n}^{\spp})$ is a compact open subgroup of $\SO(\la_{\widehat{\Z}})$ similarly as in Proposition \ref{openness}. 
Moreover, it is clear that $i_{d}(D_{d}(n))\subset D(n)$ because we have $\GSpin(\la_{d,\Z_{\ell}})\subset \GSpin(\la_{\Z_{\ell}})$ as in Lemma \ref{cliffordlemma}.
\end{enumerate}
\end{remark}

\subsection{The uniform Kuga--Satake construction} \label{subsectionuks}
In this subsection, we assume that a positive integer $n$ is sufficiently large (in our application, $n$ would be a sufficiently large power of $2$). Previous two subsections imply that there exists the following diagram of schemes over $\Q.$
\[
\begin{xy}
(0,0) *{M_{2d,D_{d}(n),\Q}^{\circ}}="A", 
(34,0) *{\Sh_{D_{d}(n)}(\SO_{\la_{d}}, X_{\SO_{\la_{d}}})}="B",
(76,0) *{\Sh_{D(n)}(\SO_{\la},X_{\SO_{\la}})}="C",
(76,14) *{\Sh_{\K_{n}^{\spp}}(\GSpin_{\la}, X_{\GSpin_{\la}})}="D",
(120,0) *{\Sh_{\K_{n}}(\GSp_{V,a}, X_{\GSp_{V,a}})}
\ar (10,0); (13.5,0) ^-{j}
\ar (54,0); (58,0) ^-{i_{d}}
\ar (76,11); (76,3.5) ^-{f}
\ar (85.4,11); (109,3.5) ^-{h}
\end{xy}
\]
Here $\Sh_{\K}(G,X)$ means the canonical model of a Shimura variety of level $\K$ associated with $(G,X)$ over $\Q$, which is the reflex field of $(G,X)$. Then, by the arguments in \cite[Section 5.5]{Rizov2010}, we can find $\delta$ which is a section of $f$ over a certain number field $E_{n}$. Indeed, as in \cite[Section 5.5]{Rizov2010}, our definition of $D(n)$ guarantees that $f$ in the above diagram induces isomorphisms between geometric connected components of the above Shimura varieties. Hence we can find a section of $f$ over a number field on which  all geometric connected components are defined.

In the following of this subsection, we fix a field $F$ containing $E_{n}$. We consider the base change from $\Q$ to $F$ of the above diagram.
\begin{equation}
\vcenter{
\xymatrix{
&&\Sh_{\K_{n}^{\spp}}(\GSpin_{\la})\ar @<0.5ex>[d]^{f} \ar[rd]^{h}& \\
M_{2d,D_{d}(n),F}^{\circ} \ar [r]^-{j} &\Sh_{D_{d}(n)}(\SO_{\la_{d}}) \ar[r]^{i_{d}}&\Sh_{D(n)}(\SO_{\la}) \ar @<0.5ex>[u]^{\delta} &\Sh_{\K_{n}}(\GSp_{V,a})   \\
}
}
\tag{$\ast$}
\end{equation}
Here, and in the following of this paper, for simplicity, we denote $(\Sh_{\K}(G,X))_{F}$ by $\Sh_{\K}(G)$. Moreover, we denote the composition $h \circ \delta \circ i_{d} \circ j$ by $\Delta_{d}$.

Remark $\ref{univab}$ implies that there exists the universal abelian scheme $\mathcal{A}$ over $\Sh_{\K_{n}}(\GSp_{V,a})$ possessing the degree $r$ polarization and the level $n$-structure. Then, for $(X,L,\nu,\alpha) \in M_{2d,D_{d}(n),F}^{\circ}(F)$ which corresponds to a morphism $t\colon \Spec F\rightarrow M_{2d,D_{d}(n),F}^{\circ}$, we can associate an abelian variety $A^{(X,L,\alpha)}$ by pulling buck $\mathcal{A}$ via $\Delta_{d} \circ t$. We will quickly recall the properties of $A^{(X,L,\alpha)}.$

\begin{definition}\label{shf}
Let $\ell$ be any prime number.
\begin{enumerate}
\item
Let $S$ be any (schematic) connected component of $\Sh_{D(n)}(\SO_{\la}),$ and $\overline{s}\rightarrow S$ be a geometric point.
Then, as in \cite[III, Remark 6.1]{Milne1990}, we can show that
\[
\plim[\K]\Sh_{\K}(\SO_{\la})\rightarrow \Sh_{D(n)}(\SO_{\la})
\]
is a Galois covering with a Galois group $D(n)$, and so we can associate the representation 
\[
\pi_{1}(S,\overline{s})\rightarrow (D(n))_{\ell}\rightarrow \SO(\la_{\Z_{\ell}}).
\]
We define $\la_{\Z_{\ell}}^{\shf}$ as  the corresponding $\Z_{\ell}$-sheaf on $\Sh_{D(n)}(\SO_{\la})$, which has a symmetric pairing structure.
\item
Similarly, we define $\la_{d,\Z_{\ell}}^{\shf}$ as the $\Z_{\ell}$-sheaf on $\Sh_{D_{d}(n)}(\SO_{\la_{d}})$ corresponding to the representation $(D_{d}(n))_{\ell}\rightarrow \SO(\la_{d,\Z_{\ell}}).$ The sheaf $\la_{d,\Z_{\ell}}^{\shf}$ has a symmetric pairing structure.
\item
Similarly, we define $V_{\Z_{\ell}}^{\shf}$ as the $\Z_{\ell}$-sheaf on $\Sh_{\K_{n}}(\GSp_{V,a})$ corresponding to the representation $(\K_{n})_{\ell} \rightarrow \GSp(V_{\Z_{\ell}}, \phi_{a})$. The sheaf $V_{\Z_{\ell}}^{\shf}$ has a symplectic pairing structure.

\end{enumerate}
\end{definition}

\begin{lemma}\phantomsection\label{morshf}
\begin{enumerate}
\item
There exists the natural injection of \'etale sheaves $\la_{d,\Z_{\ell}}^{\shf}\rightarrow i_{d}^{\ast}\la_{\Z_{\ell}}^{\shf}$ preserving the pairing. Moreover, $(\la_{d,\Z_{\ell}}^{\shf})^{\perp}$ is trivial as a $\Z_{\ell}$-sheaf.
\item
There exists the natural injection of \'etale sheaves $f^{\ast}\la_{\Z_{\ell}}^{\shf}\rightarrow \End(h^{\ast}(V_{\Z_{\ell}}^{\shf}))$, which induces a `left multiplication' on a stalk.
\end{enumerate}
\end{lemma}

\begin{proof}
For $(1)$, it is enough to show that $i_{d}\colon \la_{d,\Z_{\ell}}\rightarrow \la_{\Z_{\ell}}$ is $\pi_{1}(S,\overline{s})$-equivariant, and $(\la_{d,\Z_{\ell}})^{\perp}$ is a trivial $\pi_{1}(S,\overline{s})$-module. Here $S$ is any connected component of $\Sh_{D_{d}(n)}(\SO_{\la_{d}})$, $\overline{s}$ is a geometric point of $S$, and $\pi_{1}(S,\overline{s})$-module structure on $\la_{d,\Z_{\ell}}$, $\la_{\Z_{\ell}}$ correspond to $\la_{d,\Z_{\ell}}^{\shf}$, $i_{d}^{\ast}\la_{\Z_{\ell}}^{\shf}$. In regard to a $\Z_{\ell}$-sheaf given by a representation of adelic subgroup, a pullback of a $\Z_{\ell}$-sheaf corresponds to a pullback of a representation. Thus $\pi_{1}(S,\overline{s})$-module structure on $\la_{\Z_{\ell}}$ is given by 
\[
\pi_{1}(S,\overline{s})\rightarrow (D_{d}(n))_{\ell}\hookrightarrow (D(n))_{\ell} \hookrightarrow \SO(\la_{\Z_{\ell}}).
\] 
Hence the desired claim is clear.

For $(2)$, it is enough to show that the morphism
\[
\la_{\Z_{\ell}} \rightarrow \End(V_{\Z_{\ell}}) ; v\mapsto (z\mapsto vz)
\]
is $\pi_{1}(S,\overline{s})$-equivariant, where $S$ is any connected component of $\Sh_{\K_{n}^{\textup{sp}}}(\GSpin_{\la})$, and $\pi_{1}(S,\overline{s})$-module structure on $\la_{\Z_{\ell}}$, $V_{\Z_{\ell}}$ correspond to $f^{\ast}(\la_{\Z_{\ell}}^{\shf})$, $h^{\ast}(V_{\Z_{\ell}}^{\shf})$. By the same reason as $(1)$, these structure are given by 
\begin{align*}
\pi_{1}(S,\overline{s})\rightarrow (\K_{n}^{\textup{sp}})_{\ell} \xrightarrow{f} (D(n))_{\ell} \hookrightarrow \SO(\la_{\Z_{\ell}}),\\
\pi_{1}(S,\overline{s})\rightarrow (\K_{n}^{\textup{sp}})_{\ell} \xrightarrow{h} (\K_{n})_{\ell} \rightarrow \GSp(V_{\Z_{\ell}},\phi_{a}).
\end{align*}
Hence if we denote the first arrows of the both by $\sigma$, 
these actions are described as
\[
\gamma(v)= \sigma(\gamma)v\sigma(\gamma)^{-1}, \quad \gamma(z)=\sigma(\gamma)(z),
\]
for $\gamma \in \pi_{1}(S,\overline{s}), v\in \la_{\Z_{\ell}},$ and $z\in V_{\Z_{\ell}}.$ 
Thus the desired equivariance is clear.
\end{proof}

\begin{prop}\label{uks}
Let $\ell$ be any prime number, $t \colon \Spec F \rightarrow M_{2d,D_{d}(n),F}^{\circ}$ be the point corresponding to $(X,L, \nu, \alpha)\in M_{2d,D_{d}(n),F}^{\circ}(F)$, $A^{(X,L,\alpha)}$ be the abelian variety given by $(\Delta_{d} \circ t)^{\ast}(\mathcal{A})$,  and $\la_{\Z_{\ell},(X,L,\alpha)}$ be the $\Gal(\overline{F}/F)$-lattice identified with $(i_{d}\circ j\circ t)^{\ast}(\la_{\Z_{\ell}}^{\shf})$. Then, the following hold.
\begin{enumerate}
\item
There exists a Galois equivariant lattice embedding
\[
P^{2}_{\et}((X_{\overline{F}},L_{\overline{F}}),\Z_{\ell}(1))\subset \la_{\Z_{\ell},(X,L,\alpha)}
\]
such that $\Gal(\overline{F}/F)$ acts trivially on the orthogonal complement
\[
P^{2}_{\et}((X_{\overline{F}},L_{\overline{F}}),\Z_{\ell}(1))^{\perp}\subset \la_{\Z_{\ell},(X,L,\alpha)}.
\]
\item
The abelian variety
$A^{(X,L,\alpha)}$ has a level $n$-structure defined over $F$. Thus each $n$-torsion point of $A^{(X,L,\alpha)}$ is $F$-rational. 

\item
The abelian variety
$A^{(X,L,\alpha)}$ admits a left $C(\la)$-action over $F$, and moreover there exists an isomorphism of $\Z_{\ell}$-modules 
\[
H^{1}_{\et}(A^{(X,L,\alpha)}_{\overline{F}}, \Z_{\ell})\simeq C(\la_{\Z_{\ell},(X,L,\alpha)})
\]
which identifies the algebra 
\[
C(\la_{\Z_{\ell}})^{\textup{op}} \subset \End (H^{1}_{\et}(A^{(X,L,\alpha)}_{\overline{F}}, \Z_{\ell}))
\]
with 
\[
C(\la_{\Z_{\ell},(X,L,\alpha)})^{\textup{op}} \subset \End (C(\la_{\Z_{\ell},(X,L,\alpha)})).
\]
Here, the former inclusion of algebras is induced by the above $C(\la)$-action, and the latter is induced by the right multiplication.

\item
The left multiplication by $C(\la_{\Z_{\ell},(X,L,\alpha)})$ on the right-hand side of the isomorphism in $(3)$ induces a Galois equivariant isomorphism
\[
C(\la_{\Z_{\ell},(X,L,\alpha)}) \simeq \End_{C(\la_{\Z_{\ell}})^{\textup{op}}}(H^{1}_{\et}(A^{(X,L,\alpha)}_{\overline{F}},\Z_{\ell})).
\]
Here, the (left) $C(\la_{\Z_{\ell}})^{\textup{op}}$-module structure is induced by the left $C(\la)$-action on $A^{(X,L,\alpha)}$ as in $(3)$.
\end{enumerate}
\end{prop}

\begin{proof}
These results are essentially proved in \cite[Proposition 3.5.8]{She2017}. 

$(1)$ follows from Lemma \ref{morshf} (1) and the fact  
\[
(j\circ t)^{\ast}(\la_{d,\Z_{\ell}}^{\shf})\simeq P^{2}_{\et}((X_{\overline{F}},L_{\overline{F}}),\Z_{\ell}(1))
\]
(see \cite[Proposition 5.6 (1)]{MadapusiPera2015}).

$(2)$ is clear because the universal family  $\mathcal{A}$ admits a level $n$-structure.

Before proving $(3)$ and $(4)$, we note that for the universal abelian scheme $u\colon \mathcal{A}\rightarrow \Sh_{\K_{n}}(\GSp_{V,a})$, we have $R^{1}u_{\ast}\Z_{\ell} \simeq V_{\Z_{\ell}}^{\shf}$. 

For $(3)$, as in \cite[Section 3.10]{MadapusiPera2016}, $h^{\ast}(\mathcal{A})_{\Sh_{\K_{n}^{\textup{sp}}}(\GSpin_{\la})_{\C}}$ admits a $C(\la)$-action which corresponds to a right multiplication on the cohomology, since our definition of $h$ guarantees that the right multiplication preserves the Hodge structure. 
This action descends to $F$ by \cite[Proposition 3.11]{MadapusiPera2016} and induces a $C(\la)$-action on $A^{(X,L,\alpha)}$ with desired properties.

For $(4)$, the statement $(3)$ of this proposition and Lemma \ref{morshf} (2) implies the well-definedness and the Galois equivariance of our morphism, and it is clearly bijective.
\end{proof}

\section{Proof of the Main Theorems}
\subsection{Statements}

\begin{lemma}\label{prep}
Let $F$ be a finitely generated field over $\Q$, $R$ be a smooth algebra over $\Z$ which is an integral domain with the fraction field $F$, and $\overline{s}$ be a geometric point corresponding to an algebraic closure $\overline{F}$ over $F$. 
For any $\pi_{1}(\Spec F, \overline{s})$-module $M$ such that $\Ker (\pi_{1}(\Spec F, \overline{s}) \rightarrow \Aut(M))$ is closed, the following are equivalent.
\begin{enumerate}
\item
The $\pi_{1}(\Spec F, \overline{s})$-action on $M$ arises from a $\pi_{1}(\Spec R, \overline{s})$-action on $M$.
\item
For any height $1$ prime ideal $\mathfrak{p} \in \Spec R$,
the $\pi_{1}(\Spec F ,\overline{s})$-action on $M$ arises from a $\pi_{1}(\Spec R_{\mathfrak{p}}, \overline{s})$-action on $M$. 
\item
For any height $1$ prime ideal $\mathfrak{p} \in \Spec R$,
M is unramified at $\p$, i.e.\ if we take $\overline{v}$ which is an extension of valuation $\p$ to $\overline{F}$, the inertia group $I_{\overline{v}}$ acts trivially on $M$.
\end{enumerate}
When $M$ satisfies the above equivalent conditions, we say $M$ is unramified over $\Spec R$. 
\end{lemma}
\begin{proof}
First, we recall that $\pi_{1}(\Spec F, \overline{s}) \rightarrow \pi_{1}(\Spec R, \overline{s})$ is surjective and its kernel is identified with $\Gal(\overline{F}/F_{R}^{\textup{ur}})$, where $F_{R}^{\textup{ur}}$ is the composite of finite extensions $E/F$ which are unramified over $\Spec R$ (\cite[Proposition 3.3.6]{Fu2015}). Here, we say $E/F$ is unramified over $\Spec R$ if the normalization of $\Spec R$ in $E$ is unramified over $\Spec R$. Same results hold for $R_{\p}$.  

$(1)\Leftrightarrow (2)$
By the assumption on $M$, it suffices to show that
\[
\Ker(\pi_{1}(\Spec F, \overline{s})\rightarrow \pi_{1}(\Spec R, \overline{s}))
\]
is generated by 
\[
(\Ker(\pi_{1}(\Spec F, \overline{s})\rightarrow \pi_{1}(\Spec R_{\p}, \overline{s})))_{\p}
\]
as a topological group.
By the above remark, it is enough to show that $F^{\textup{ur}}_{R}$ = $\bigcap_{\textup{ht}(\p)=1}F^{\textup{ur}}_{R_{\p}}$. 
The inclusion
$F^{\textup{ur}}_{R} \subset \bigcap_{\textup{ht}(\p)=1}F^{\textup{ur}}_{R_{\p}}$ is obvious, and another direction follows from the Zariski--Nagata purity.

$(2)\Leftrightarrow (3)$
By the assumption on $M$, it suffices to show that $\Ker(\pi_{1}(\Spec F, \overline{s})\rightarrow \pi_{1}(\Spec R_{\p}, \overline{s}))$ is generated by $(I_{\overline{v}})_{\overline{v} \textup{ over } \p}$ as a topological group, but it follows from the above remark.
\end{proof}
\begin{remark}
The condition `$M$ is unramified at $\p$' does not depend on a choice of $\overline{v}$. Indeed, for each $\p$, the inertia group $I_{\overline{v}}$ is determined by $\p$ up to conjugation in $\Gal(\overline{F}/F).$
\end{remark}

The following are the statements of results of this section (for more generalized statements, see Theorem \ref{main}).
\begin{theorem}\label{genp}
Let $F$ be a finitely generated field over $\Q$, $R$ be a smooth algebra over $\Z$ which is an integral domain with the fraction field $F$, and $d$ be a positive integer. Then, the set
\[
\Shaf(F,R,d) \coloneqq
\left\{
(X, L)
\left|
\begin{array}{l}
X \colon\textup{K3 surface over }F, \\
L \in \Pic_{X/F}(F)\colon \textup{primitive ample},\\
H^{2}_{\et}(X_{\overline{F}}, \Q_{2}) \colon \textup{unramified over } \Spec R,\\
\deg L=2d
\end{array}
\right.
\right\} /F\textup{-isom}.
\]
is finite.
\end{theorem}

\begin{theorem}\label{genunp}
Let $F$ be a finitely generated field over $\Q$, and $R$ be a smooth algebra over $\Z$ which is an integral domain with the fraction field $F$. Then, the set
\[
\Shaf(F,R)\coloneqq
\left\{
X\left|
\begin{array}{l}
X \colon\textup{K3 surface over }F,\\
H^{2}_{\et}(X_{\overline{F}},\Q_{2})\colon \textup{unramified over } \Spec R
\end{array}
\right.
\right\}/F\textup{-isom}
\]
is finite.
\end{theorem}

\begin{remark}\label{avoidl}
For a non-empty open subscheme $\Spec(R')\subset \Spec(R)$, the finiteness of $\Shaf(F, R', d)$ (resp.\ $\Shaf(F, R')$) clearly implies the finiteness of $\Shaf(F, R, d) $ (resp.\ $\Shaf(F, R)$). 
Thus, to prove Theorem \ref{genp} and Theorem \ref{genunp}, we may assume $1/2\in R$. 
Note that it is equivalent to say that the residual characteristic at any point of $\Spec R$ is different from $2$.
\end{remark}

\subsection{Proof of Theorem \ref{genp}}
In this subsection, we use the same notation as Theorem \ref{genp}, unless otherwise noted. First, for using the Kuga--Satake construction, we will replace $F$ by an appropriate finite extension of it to provide a level structure on $(X,L)\in \Shaf(F, R, d)$. The following lemma is essential for justifying this replacement.

\begin{lemma}\label{twists}
Let $E/F$ be a finite extension, $X_{0}$ be a K3 surface over $F$, and
$L_{0}\in \Pic_{X_{0}/F}(F)$ be a polarization. Then, the set
\[
\left\{
(X,L)
\left|
\begin{array}{l}
X \textup{ is a K3 surface over }F,\\
L \in \Pic_{X/F}(F)\colon \textup{ample},\\
(X_{E}, L_{E}) \simeq_{E} (X_{0,E},L_{0,E})
\end{array}
\right.
\right\}/ F\textup{-isom}
\]
is finite.
\end{lemma}
\begin{proof}
Taking the Galois closure of $E$ in $\overline{F}$ , we may assume that $E/F$ is a Galois extension. Then we can identify this set with the Galois cohomology group $H^{1}(\Gal(E/F), \Aut_{E}(X_{0}, L_{0}))$.
The finiteness of this set follows from \cite[Chapter 5, Proposition 3.3]{Huybrechts2016}.
\end{proof}

\begin{lemma}[{cf.\ \cite[Lemma 8.4.1]{Andre1996}}]\label{andre}
Let $X$ be a K3 surface over $F$, $L\in \Pic_{X/F}(F)$ be a primitive polrization of degree $2d$ on $X$ over $F$, and $n$ be a positive integer. 
We put 
\[
W_{\widehat{\Z}} \coloneqq P^{2}_{\et}((X_{\overline{F}},L_{\overline{F}}),\widehat{\Z}(1)).
\] 
Let 
\[
\rho \colon \Gal(\overline{F}/F)\rightarrow \algO(W_{\widehat{\Z}})
\]
be the natural Galois representation. 
Fix an isometry 
\[
i_{(X,L)}\colon \la_{K3,\widehat{\Z}}\simeq H^{2}_{\et}(X_{\overline{F}},\widehat{\Z}(1)),
\]
which restricts to an isometry
$\la_{K3} \simeq H^{2}(X(\C),\Z(1))$,
and which sends $v_{d}$ to $\ch_{\widehat{\Z}}(L)$ $($see Remark \ref{remlattice}$)$. Using $i_{(X,L)}$, we identify $D_{d}(n)$ with a compact open subgroup of $\SO(W_{\widehat{\Z}})$.  Then, for any finite extension $E/F$, we have
\[
\rho(\Gal(\overline{F}/E))\subset D_{d}(n)
\Leftrightarrow  \rho_{\ell}(\Gal(\overline{F}/E))\subset (D_{d}(n))_{\ell} \textup{ for every } \ell \mid 2n.
\]
\end{lemma}
\begin{proof}
In the following, we identify $\SO(\la_{d,\widehat{\Z}})$ with $\SO(W_{\widehat{\Z}})$ vie $i_{(X,L)}$.  This lemma is essentially shown in \cite[Lemma 8.4.1]{Andre1996}. Andr$\e$ shows the following claim in the proof of \cite[Lemma 8.4.1]{Andre1996}, using specialization arguments and the Weil conjecture.

\begin{claim*}
If there exists a prime number $\ell$ such that $\rho_{\ell}(\Gal(\overline{F}/E))\subset \SO(W_{\Z_{\ell}})$,
then $\rho(\Gal(\overline{F}/E))\subset \SO(W_{\widehat{\Z}})$.
\end{claim*}
Andr$\e$ states that the above claim implies the following result.
\[
\rho(\Gal(\overline{F}/E))\subset D_{d}(n)
\Leftrightarrow  \rho_{\ell}(\Gal(\overline{F}/E))\subset (D_{d}(n))_{\ell} \textup{ for every } \ell \mid 2dn.
\]
Indeed, for $\ell\nmid 2dn$, we have $(D_{d}(n))_{\ell}=\SO(\la_{d,\Z_{\ell}})$.

More generally, for $\ell \nmid 2n$, we have $(D_{d}(n))_{\ell}=(D_{d})_{\ell}$. (See Corollary \ref{D_{d}(n)}). Therefore, to generalize Andr$\e$'s result to our lemma, it is enough to show that if $\rho_{\ell}(\Gal(\overline{F}/E))\subset \SO(W_{\Z_{\ell}}),$ then $\rho_{\ell}(\Gal(\overline{F}/E))\subset (D_{d})_{\ell}$. However, since $\Gal(\overline{F}/F)$ stabilizes $\ch_{\Z_{\ell}}(L)$, it follows from our description of the discriminant kernel
\[
(D_{d})_{\ell}=\{\tilde{g}_{\ell}\in\SO(H^{2}_{\et}(X_{\overline{F}}, \Z_{\ell}))\mid\tilde{g}_{\ell}(\ch_{\Z_{\ell}}(L))=\ch_{\Z_{\ell}}(L)\},
\] 
which follows from Proposition \ref{disc}.
\end{proof}

In the rest of this section, fix a positive integer $n$ which is a sufficiently large power of $2$.

\begin{prop}\label{redp}
To prove Theorem \ref{genp}, it is enough to show that
\[
\Shaf'(F, R, d)\coloneqq
\left\{
(X, L)
\left|
\begin{array}{l}
X \colon \textup{K3 surface over }F, \\
L \in \Pic_{X/F}(F)\colon \textup{ primitive ample},\\
H^{2}_{\et}(X_{\overline{F}}, \Q_{2}) \colon \textup{unramified over } \Spec R,\\
\deg L=2d,\\
(X, L) \textup{ admits a } D_{d}(n) \textup{-level structure}
\end{array}
\right.
\right\} /F\textup{-isom}
\]
is a finite set for any $F, R, d$ as in Theorem \ref{genp}. Moreover, if we fix a number field $F'$, it suffices to show only in the case where $F\supset F'$ and $1/2\in R$.

Here, `$(X, L)$ admits a $D_{d}(n)$-level structure' means that
there exists an element $(X, L, \nu_{(X, L)}, \alpha_{(X, L)})$ in $M_{2d,D_{d}(n),F}^{\circ}(F)$.

\end{prop}

\begin{proof}
We should prove the finiteness of $\Shaf(F,R,d).$ By Remark \ref{avoidl}, we may assume $1/2 \in R$ (so the Tate twist $\otimes \Z_{2}(1)$ does not affect the unramifiedness over $\Spec R$, see Lemma \ref{prep}).

First, we will show that there exists a finite extension $E/F$ such that for any $(X,L)\in \Shaf(F,R,d),$ the pair $(X_{E}, L_{E})$ admits a $D_{d}(n)$-level structure. We fix $(X,L) \in \Shaf(F,R,d)$ and $i_{(X,L)}$, moreover we use the same identification as in Lemma \ref{andre}.
Let 
\[
\overline{\rho}_{2} \coloneqq \overline{\rho}_{(X,L),2}\colon \pi_{1}(\Spec R, \overline{s})\rightarrow \algO(P^{2}_{\e t}((X_{\overline{F}},L_{\overline{F}}), \Z_{2}(1)))
\]
be the representation induced by 
\[
\rho \coloneqq \rho_{(X,L)}\colon \Gal(\overline{F}/F)\rightarrow \algO(P^{2}_{\e t}((X_{\overline{F}},L_{\overline{F}}),\widehat{\Z}(1))).
\]
The inverse image $\overline{\rho}^{-1}_{2}((D_{d}(n))_{2})$ is a finite index subgroup, so we can associate a pointed finite $\e$tale cover $\Spec \tilde{R} \rightarrow \Spec R.$ 
Then we have $\overline{\rho}_{2}(\pi_{1}(\Spec \tilde{R},\overline{s}))\subset (D_{d}(n))_{2}.$ 
The former is equal to $\rho_{2}(\Gal(\overline{F}/\Frac(\tilde{R})))$, and by Lemma \ref{andre}, we can get the $D_{d}(n)$-level structure on 
$(X_{\Frac(\tilde{R})},L_{\Frac(\tilde{R})})$ by $i_{(X,L)}$.

Here, note that 
\begin{align*}
[\pi_{1}(\Spec R, \overline{s})\colon\overline{\rho}^{-1}_{2}((D_{d}(n))_{2})]\leq C_{d}\coloneqq [\algO(\la_{d,\Z_{2}})\colon (D_{d}(n))_{2}],
\end{align*}
where $C_{d}$ is independent of $(X,L)$ and $i_{(X,L)}$.
By the analogue of the Hermite--Minkowski theorem \cite[Proposition 2.3, Theorem 2.9]{Harada2009}, the family of subsets
\[
\mathcal{C}\coloneqq
\{H\subset \pi_{1}(\Spec R, \overline{s}) \colon \textup{open subgroup} \mid [\pi_{1}(\Spec R, \overline{s})\colon H]\leq C_{d}
\}
\]
is finite, therefore 
\[
H_{0} \coloneqq \bigcap_{H\in \mathcal{C}}H
\]
is an open subgroup.
Let $\Spec \tilde{R}_{0}\rightarrow \Spec R$ be the corresponding pointed finite $\e$tale covering, then by the above argument, we can get a $D_{d}(n)$-level structure on $(X_{\Frac(\tilde{R}_{0})},L_{\Frac(\tilde{R}_{0})})$. Hence we now get a desired finite extension $E\coloneqq \Frac(\tilde{R}_{0})$. 

Thus, by using the assumption for $\Shaf'(E, \tilde{R}_{0}, d)$ and Lemma \ref{twists}, we can show the finiteness of $\Shaf(F,R,d)$. Note that the latter statement is clear by Lemma \ref{twists}.
\end{proof}

The following proposition is essentially known by Andr\'{e} (\cite{Andre1996}), and one can prove it as a corollary of the theory of potentially good loci of Shimura variety (\cite{Imai2020}). 

\begin{prop}\label{good}
Assume $F\supset E_{n}$, where $E_{n}$ is as in $3.3$.
For $(X,L,\nu,\alpha)\in M_{2d,D_{d}(n),F}^{\circ}(F)$, let $A^{(X,L,\alpha)}$ be the Kuga--Satake abelian variety as in Proposition \ref{uks}. Let $R$ be a smooth algebra over $\Z$ which is an integral domain with the fraction field $F$, and assume $1/2 \in R$.
Assume that $H^{2}_{\et}(X_{\overline{F}},\Q_{2})$ is unramified over $\Spec R$ $($its Tate twists are unramified too, because $1/2 \in R)$.  Then, for any height 1 prime ideal $\p \in \Spec R$, the abelian variety $A^{(X,L,\alpha)}$ has good reduction at $\p$.
\footnote{If $n$ is a sufficiently large power of $\ell$, then the same argument work with $\ell$ in place of $2$.}
\end{prop}

\begin{proof}
We will follow the proof by Andr\'{e} (\cite[Lemma 9.3.1]{Andre1996}).
\footnote{The referee taught me another quick way of seeing this proposition. By the construction, the Galois representation on $H^{1}_{\et}(A^{(X,L,\alpha)}_{\overline{F}},\Q_{2})$ has to factor through $\GSpin(PH^{2}_{\et}(X_{\overline{F}},\Q_{2}))$ and hence the inertia representation has to factor through the center $\Q_{2}^{\times}$ and in fact through $\Z_{2}^{\times}$ by the compactness. Moreover, by the level assumption, it factors through $1+4\Z_{2}$, so it finishes proof since $1+4\Z_{2}$ has no non-trivial quasi-unipotent elements.}
By the N$\e$ron--Ogg--Shafarevich criterion for abelian varieties (it is true whether a residue field is perfect or not), it is enough to show that $H^{1}_{\et}(A^{(X,L,\alpha)}_{\overline{F}},\Z_{2})$ is unramified at $\p$ (here we use $1/2\in R$).
Let $\overline{v}$ be an extension on $\overline{F}$ of the valuation $\p$, and $\varphi \colon I_{\overline{v}}\rightarrow \Aut(H^{1}_{\et}(A^{(X,L,\alpha)}_{\overline{F}},\Z_{2}))$ be a restriction of the Galois representation. Since $C(\la)$-action on $A^{(X,L,\alpha)}$ is defined over $F$, \textup{ for any } $\gamma \in I_{\overline{v}},$ we have
\[
\varphi (\gamma)\in \End_{C(\la_{\Z_{2}})^{\textup{op}}}(H^{1}_{\et}(A^{(X,L,\alpha)}_{\overline{F}},\Z_{2}))\simeq C(\la_{\Z_{2},(X,L,\alpha)})
\]
(see Proposition \ref{uks} (3), (4)). Thus we also denote its image by $\varphi(\gamma)\in C(\la_{\Z_{2},(X,L,\alpha)})$.

On the other hand, $I_{\overline{v}}$ acts trivially on $P^{2}_{\et}((X_{\overline{F}},L_{\overline{F}}),\Z_{2}(1))$ by our assumptions (see Lemma \ref{prep}), and moreover acts trivially on $P^{2}_{\et}((X_{\overline{F}},L_{\overline{F}}),\Z_{2}(1))^{\perp}\subset \la_{\Z_{2},(X,L,\alpha)}$ by Proposition \ref{uks} (1), thus 
$\gamma(c)=c$ for any $\gamma \in I_{\overline{v}}$ and $c \in C(\la_{\Z_{2},(X,L,\alpha)})$.
By Proposition \ref{uks} (4), we have $\gamma(z\mapsto cz)=(z\mapsto cz)$ in $\End_{C(\la_{\Z_{2}})^{\textup{op}}}(H^{1}_{\et}(A^{(X,L,\alpha)}_{\overline{F}},\Z_{2}))$, where the left-hand side is $(z\mapsto \varphi(\gamma)c\varphi(\gamma)^{-1}z)$. This implies $\varphi(\gamma)$ is contained in the center of $C(\la_{\Q_{2},(X,L,\alpha)})$, which is a reduced algebra.

The Raynaud semi-abelian reduction criterion \cite[Expos\'e IX, Proposition 4.7]{SGA7-I} and Proposition \ref{uks} (2) imply that $A^{(X,L,\alpha)}$ has semi-abelian reduction at $\p$ (i.e.\ $A^{(X,L,\alpha)}$ extends to a semi-abelian scheme over $\Spec R_{\p}$). Here, we use that $n\geq 3$ is a power of $2$, and the residual characteristic of $\p$ is not $2$. Thus for any $\gamma \in I_{\overline{v}}$, $\varphi(\gamma)$ is a unipotent element of a reduced algebra, it is identity. Hence it finishes the proof.
\end{proof}

We now complete the proof of Theorem \ref{genp}. By Proposition \ref{redp}, it is enough to show the finiteness of $\Shaf'(F,R,d)$ when $F\supset E_{n}$ and $1/2\in R$. Here, we take $E_{n}$ as in $3.3$. In the following, we identify $(X,L)\in \Shaf'(F,R,d)$ with $(X, L, \nu, \alpha)\in M_{2d,D_{d}(n),F}^{\circ}(F)$ by choosing a level structure. 
Hence for $(X,L,\nu,\alpha) \in \Shaf'(F,R,d),$ we can associate $A^{(X,L,\alpha)}$, and since in the diagram $(\ast)$ of 3.3, each fiber of $i_{d}$ is finite and $h$ is injective (because they are induced by an embedding of Shimura data), it suffices to show the finiteness of $\Delta_{d} (\Shaf'(F, R, d)).$ 
The image $\Delta_{d}(X,L,\nu,\alpha)$ corresponds to $A^{(X,L,\alpha)}$ with their degree $r$ polarization and level $n$-structure. 
However, by Proposition \ref{good}, the abelian variety $A^{(X,L,\alpha)}$ has good reduction at any height $1$ prime of $\Spec R$, so this set is finite by \cite[Satz 6]{Faltings1983} (for finitely generated fields of characteristic $0$, see \cite[VI, \S 1, Theorem 2]{Faltings1992}).

\subsection{Proof of Theorem \ref{genunp}}
In this subsection, we use the same notation as in Theorem \ref{genunp}, unless otherwise noted. The strategy is the same as \cite{She2017}, i.e.\ we use Theorem \ref{genp} for reducing the problem to the finiteness of Picard lattices, and use the uniform Kuga--Satake maps for associating $\Shaf(F,R)$ with a finite set of abelian varieties.

\begin{lemma}[{cf.\ \cite[Corollary 4.1.3]{She2017}}]\label{Pic}
For any $X_{0}\in \Shaf(F, R)$, 
there exist only finitely many $X\in \Shaf(F,R)$ whose Picard lattice $\Pic_{X/F}(F)$ is isometric to the Picard lattice $\Pic_{X_{0}/F}(F)$.
\end{lemma}
\begin{proof}
As in \cite[Proposition 4.1.2]{She2017}, a K3 surface $X$ over $F$ admits a primitive polarization whose degree bounded by a constant depending only on the isometry class of $\Pic_{X/F}(F)$. Hence this lemma follows from Theorem \ref{genp}.
\end{proof}

\begin{lemma}[{cf.\ \cite[Lemma 4.1.4]{She2017}}]\label{twistunp}
Let $E /F$ be a finite extension. For any $X_{0}\in \Shaf(F, R)$, the set
\[
\{X\in \Shaf(F,R) \mid X_{E}\simeq_{E} X_{0,E}
\}
\]
is finite.
\end{lemma}

\begin{proof}
Taking a Galois closure, we may assume $E/F$ is a Galois extension.
By Lemma \ref{Pic}, 
it suffices to show the finiteness of isometry classes of Picard lattices $\Pic_{X/F}(F)$ associated with the considering set. 
Note that $\Pic_{X/F}(E)^{\Gal(E/F)}= \Pic_{X/F}(F)$ and $\Pic_{X/F}(E)$ is isometric to $\Pic_{X_{0}/F}(E)$.
Since the set of conjugacy classes of subgroups of $\algO(\Pic_{X_{0}/F}(E))$ with the order $[E\colon F]$ is finite by \cite[Section 5, (a)]{Borel1963}, the desired finiteness follows.
\end{proof}

\begin{prop}\label{redunp}
Recall that we fixed a positive integer $n$ which is a power of $2$.
To show Theorem \ref{genunp}, it is enough to show that
\[
\Shaf'(F,R)
\coloneqq
\left\{X
\left|
\begin{array}{l}
X\colon \textup{K3 surface over } F, \\
H^{2}_{\et}(X_{\overline{F}},\Q_{2})\colon \textup{unramified over }\Spec R,\\
\textup{there exists } d_{X},  L_{X}, \nu_{X}, \alpha_{X} \textup{ such that}  \\
(X, L_{X}, \nu_{X}, \alpha_{X}) \in M_{2d_{X},D_{d_{X}}(n),\Q}^{\circ}(F)
\end{array}
\right.
\right\}
/F\textup{-isom}
\]
is a finite set for any $(F, R)$ as in Theorem \ref{genunp}. Moreover, if we fix a number field $F'$, it suffices to show only in the case where $F\supset F'$ and $1/2 \in R$.
\end{prop}

\begin{proof}
The proof is similar to Proposition \ref{redp}, but we need more precise evaluation since we should discuss all degrees simultaneously.

As in the proof of Proposition \ref{redp}, it is enough to show the finiteness of $\Shaf(F,R)$ with $1/2 \in R$. 
First, remark that every K3 surface over $F$ admits some primitive polarization over $F$.
Therefore, for any $X \in \Shaf(F,R)$, we can associate a primitive polarization $L_{X}$. Let $2d_{X}$ be the degree of $L_{X}$.
We will show that there exists a finite extension $E/F$ such that for any $X \in \Shaf(F,R)$, the pair $(X_{E},L_{X,E})$ admits a $D_{d_{X}}(n)$-level structure. For each $X \in \Shaf(F,R),$ we fix $i_{(X,L_{X})}$ as in Lemma \ref{andre}, and we use the notation $\overline{\rho}_{2}\coloneqq \overline{\rho}_{(X,L_{X}),2}$ in the same sense as in Proposition \ref{redp}. To get a desired extension, we should replace the bound $C_{d_{X}}$ in the proof of Proposition \ref{redp} by a bound which is independent of $X$.
For $\Gamma\coloneqq \overline{\rho}_{2}(\pi_{1}(\Spec R, \overline{s}))$, we have
\begin{align*}
[\Gamma \colon \Gamma \cap (D_{d_{X}}(n))_{2}]
&=[\Gamma \colon \Gamma \cap \SO(\la_{d,\Z_{2}})]\cdot [\Gamma \cap \SO(\la_{d_{X},\Z_{2}}) \colon \Gamma \cap(D_{d_{X}}(n))_{2}]\\
&\leq 2N \cdot n^{(2^{20})}.
\end{align*}

Here, we use $\Gamma\cap \SO(\la_{d,\Z_{2}}) \subset \Gamma\cap (D_{d})_{2}$ (follows from Proposition \ref{disc}, see the proof of Lemma \ref{andre}), and Corollary \ref{D_{d}(n)}. We note that this bound is independent of $X, L_{X},$ and $i_{(X,L_{X})}$. Hence replacing $C_{d}$ by $2N \cdot n^{(2^{20})}$ in the arguments in the proof of Proposition \ref{redp}, we get a pointed finite $\e$tale covering $\Spec \tilde{R}_{0}\rightarrow \Spec R$, whose fraction field $E$ satisfies the desired property.
Thus, by using the assumption for $\Shaf'(E,\tilde{R}_{0})$ and Lemma \ref{twistunp}, we get the finiteness of $\Shaf(F,R).$ The latter statement is clear by Lemma \ref{twistunp}.
\end{proof}

\begin{definition}[{\cite[Definition 4.1.10]{She2017}}]
Let $F$ be a subfield of $\C$, $X$ be a K3 surface over $F$, and $\ell$ be any prime number.
We define (relative) \emph{transcendental lattices} by 
\[
\begin{aligned}
T(X)&\coloneqq \Pic_{X/F}(F)^{\perp}\subset H^{2}(X(\C),\Z(1)),\\
T(X)_{\Z_{\ell}}&\coloneqq \Pic_{X/F}(F)^{\perp}\subset H_{\et}^{2}(X_{\overline{F}},\Z_{\ell}(1)),\\
T(X)_{\widehat{\Z}}&\coloneqq \Pic_{X/F}(F)^{\perp} \subset H_{\et}^{2}(X_{\overline{F}}, \widehat{\Z}(1)).
\end{aligned}
\]
Here we omit the Chern class map. 
\end{definition}

\begin{remark}[{cf.\ \cite[Corollary 4.1.13]{She2017}}]\label{samedisc}
Recall that $M\coloneqq H^{2}(X(\C),\Z(1))\simeq \la_{K3}$ is unimodular, and $N\coloneqq \Pic_{X/F}(F)$ is a primitive sublattice. In this situation, one can verify a canonical isomorphisms 
\[
N^{\vee}/N\simeq M/(N+N^{\perp}) \simeq (N^{\perp})^{\vee}/N^{\perp}.
\]
Thus we get $\disc(\Pic_{X/F}(F)) = \disc(T(X))$.
\end{remark}

\begin{lemma}[{cf.\ \cite[Proposition 4.1.11]{She2017}}]\label{tate}
For $(X,L,\nu,\alpha)\in M_{2d,D_{d}(n),F}^{\circ}(F)$ and any prime number $\ell$, we have
\[
(\la_{\Z_{\ell},(X,L,\alpha)}^{\Gal(\overline{F}/F)})^{\perp}= T(X)_{\Z_{\ell}}.
\]
Here, the orthogonal complement of the left-hand side is taken in $\la_{\Z_{\ell},(X,L,\alpha)}$, and the above equality is as a sublattice of $P^{2}_{\et}((X_{\overline{F}},L_{\overline{F}}), \Z_{\ell})$. 
\end{lemma}
\begin{proof}
First, we can show that  
\[
T(X)_{\Z_{\ell}}=(P^{2}_{\et}((X_{\overline{F}},L_{\overline{F}}),\Z_{\ell}(1))^{\Gal(\overline{F}/F)})^{\perp}
\]
(the orthogonal complement of the right-hand side is taken in $P^{2}_{\et}((X_{\overline{F}},L_{\overline{F}}),\Z_{\ell}(1))$). 
Indeed, since the both sides of this equality are primitive in $P^{2}_{\et}((X_{\overline{F}},L_{\overline{F}}),\Z_{\ell}(1))$, it suffices to show this equality after inverting $\ell$, which follows directly from the Tate conjecture over $F$ (\cite[Theorem 5.6 (a)]{Tate1994}).

Hence we have to show that 
\[
(P^{2}_{\et}((X_{\overline{F}},L_{\overline{F}}),\Z_{\ell}(1))^{\Gal(\overline{F}/F)})^{\perp} = (\la_{\Z_{\ell},(X,L,\alpha)}^{\Gal(\overline{F}/F)})^{\perp}
\]
(remark that the $\perp$ in both sides have different meaning). However, since the both sides are primitive in $\la_{\Z_{\ell},(X,L,\alpha)}$, we may invert $\ell$ for showing this equality, so it follows obviously from Proposition \ref{uks} (1). 
\end{proof}

Let us complete the proof of Theorem \ref{genunp}.
As in the previous subsection, by Proposition \ref{genunp}, it suffices to show the finiteness of $\Shaf'(F, R)$ when $F\supset E_{n}$ and $1/2\in R$.
By Lemma \ref{Pic} and the fact \cite[ch. 9, Theorem 1.1]{Cassels1982} which asserts the finiteness of isometry classes of lattices with bounded rank and discriminant, it is enough to show that $\disc(\Pic_{X/F}(F))$ $(X\in \Shaf'(F, R))$ is bounded.
Using Remark \ref{samedisc}, we can reduce the problem to the finiteness of $\{T(X)_{\widehat{\Z}}\mid X\in \Shaf'(F,R) \}/$isometry.

For $X\in \Shaf'(F, R)$, we choose an element 
\[
(X,L_{X},\nu_{X}, \alpha_{X}) \in M_{2d_{X},D_{d_{X}}(n),F}^{\circ}(F).
\]
Then, by Proposition \ref{good} and \cite[Theorem 1]{Zarhin1985} (for finitely generated fields of characteristic $0$, see \cite[VI, \S 1, Theorem 2]{Faltings1992}), the subset
\[
\{\Delta_{d_{X}}(X, L_{X},\nu_{X},\alpha_{X}) \mid X \in \Shaf'(F, R) \} \subset \Sh_{\K_{n}}(\GSp_{V,a})(F)
\]
is finite. We denote them by $t_{1}, \ldots t_{m}$, and we put
\[
\Shaf'(F,R)_{i}\coloneqq \{X\in \Shaf'(F,R)\mid \Delta_{d_{X}}(X,L_{X},\nu_{X},\alpha_{X})=t_{i}\}.
\]
Thus, the desired finiteness follows from the following lemma.

\begin{lemma}
The $\widehat{\Z}$-lattices
$T(X)_{\widehat{\Z}}$ $(X\in \Shaf'(F, R)_{i})$ are isometric to each other.
\end{lemma}
\begin{proof}
By Lemma \ref{tate}, it suffices to show that $\la_{\Z_{\ell},(X,L_{X},\alpha_{X})}$ $(X\in \Shaf'(F, R)_{i})$ is unique up to a $\Gal(\overline{F}/F)$-equivariant isometry, for any $\ell$. We denote the lift of $t_{i}$ on $\Sh_{D(n)}(\SO_{\la})$ via $h \circ \delta$ (it exists by the definition of $t_{i}$, and it is unique because $h \circ \delta$ is injective) by $\tilde{t}_{i}.$ 
Recall that we have the $\et$ale sheaf $\la_{\Z_{\ell}}^{\shf}$, which has a symmetric pairing structure, so we get the $\Gal(\overline{F}/F)$-lattice $\tilde{t}_{i}^{\ast}(\la_{\Z_{\ell}}^{\shf})$, which depends only on $t_{i}$.
By our construction of $\la_{\Z_{\ell},(X,L,\alpha)}$ in Proposition \ref{uks}, for any $X \in \Shaf'(F, R)_{i},$ the $\Gal(\overline{F}/F)$-lattice $\la_{\Z_{\ell},(X,L_{X},\alpha_{X})}$ is none other than $\tilde{t}_{i}^{\ast}(\la_{\Z_{\ell}}^{\shf})$ and it finishes the proof.
\end{proof}

\section{\texorpdfstring{$\ell$}{l}-independence}
In this section, we prove $\ell$-independence of the unramifiedness for completing the proof of the main theorem. 
\begin{lemma}\label{lindep}
Let $K$ be a Henselian discrete valuation field, $k$ be the residue field of $K$, $p$ be the characteristic of $k$, and $X$ be a smooth proper surface $X$ over $K$. Then, the following are equivalent.
\begin{itemize}
\item[$(a)$]
The $\Gal(\overline{K}/K)$-representation on $H^{2}_{\et}(X_{\overline{K}},\Q_{\ell})$ is unramified for some $\ell\neq p$.
\item[$(b)$]
The $\Gal(\overline{K}/K)$-representation on $H^{2}_{\et}(X_{\overline{K}},\Q_{\ell})$ is unramified for all $\ell\neq p$.
\end{itemize}
Moreover, if $K$ is a complete discrete valuation field of mixed characteristic $(0,p)$ with the perfect residue field $k$ and $X$ is a K3 surface over $K$, then $(a)$ $(\Leftrightarrow (b))$ is equivalent to the following.
\begin{itemize}
\item[$(c)$]
The $\Gal(\overline{K}/K)$-representation on $H^{2}_{\et}(X_{\overline{K}},\Q_{p})$ is crystalline.
\end{itemize}
\end{lemma}

\begin{remark}\phantomsection\label{remlindep}
\begin{enumerate}  
\item
In fact, for K3 surfaces, Lemma \ref{lindep} is already mentioned by Madapusi Pera in \cite[Remark 4.3]{Matsumoto2015a} (using the Kuga--Satake construction, we can reduce the problem to the case of abelian varieties). 
We note that such arguments also appeared in \cite{Imai2020}. 
So we will prove only the $\ell$ versus $\ell'$ part for general smooth proper surfaces as a corollary of Matsumoto's $\ell$-independence result (\cite[Theorem 3.3 (2)]{Matsumoto2016}).
\item
If we assume that $X$ admits a Kulikov model after a finite extension of $K$, then Lemma \ref{lindep} is known as a corollary of a good reduction criterion for K3 surfaces
(see \cite[Theorem 1.1]{Chiarellotto2019} for example). 
\item 
For any smooth proper surface $X$ over $K$, one can easily prove the similar assertion for $H^{i}_{\et}$ $(i\neq2)$. Indeed, the case of $i=0, 4$ is trivial. Moreover, for $i=1,3$, by using the Picard variety, we can reduce the problem to the case of abelian varieties. Therefore, it follows from the N{\e}ron--Ogg--Shafarevich criterion for abelian varieties (and its crystalline analogue \cite[Theorem 1]{Coleman1999}).
\end{enumerate}
\end{remark}

\subsection{Proof of Lemma \ref{lindep}}
In this subsection, we prove the Lemma \ref{lindep}.
As in Remark \ref{remlindep}, it is enough to show the equivalence $(a) \Leftrightarrow (b)$ in Lemma \ref{lindep}.
Let $K$ be a Henselian discrete valuation field, $k$ be the residue field of $K$, $p$ be the characteristic of $k$, and $X$ be a smooth proper surface over $K$.

First, we recall the definition of the monodromy operator.
\begin{definition}\label{defmonodromyl}
Let $\ell$ be a prime number different from $p$.
Consider the representation
\[
\rho_{\ell}\colon \Gal (\overline{K}/K)\rightarrow \GL(H^{2}_{\et}(X_{\overline{K}},\Q_{\ell})).
\]
By Grothendieck's monodromy theorem, there exists an open subgroup of the inertia subgroup $J\subset I_{K}$ and the nilpotent operator 
\[
N_{\ell}\colon H^{2}_{\et}(X_{\overline{K}},\Q_{\ell})(1)\rightarrow H^{2}_{\et}(X_{\overline{K}},\Q_{\ell})
\]
such that for all $\sigma \in J$,
we have
$\rho_{\ell}(\sigma)=\exp(t_{\ell}(\sigma)N_{\ell})$, where $t_{\ell}\colon I_{K}\rightarrow \Z_{\ell}(1)$ is a natural projection. By fixing an isomorphism $\Q_{\ell}(1)\simeq \Q_{\ell}$, we regard $N_{\ell}$ as a linear endomorphism of $H^{2}_{\et}(X_{\overline{K}},\Q_{\ell})$, which is called the
\textit{monodromy operator}.
\end{definition}
\begin{remark}\label{remmonodromy}
By the definition, $N_{\ell}$ does not change if we replace $K$ by a finite extension of it.
\end{remark}

The following lemma is an elementary fact about $\ell$-adic representations.
\begin{lemma}\label{tr and monodromy}
The following are equivalent.
\begin{enumerate}
\item
The $\ell$-adic representation
$\rho_{\ell}$ is unramified.
\item
$N_{\ell}=0$ and $\tr(\rho_{\ell}(\sigma))= \dim (H^{2}_{\et}(X_{\overline{K}},\Q_{\ell}))$ for any $\sigma \in I_{K}$.
\end{enumerate}
\end{lemma}

\begin{proof}
$(1)\Rightarrow (2)$ is trivial. Therefore we prove the opposite direction.
By the definition of the monodromy operator, we have $\rho_{\ell}(g)=1$ for any $g \in J$, where $J$ is an open subgroup of $I_{K}$. Hence for any $\sigma \in I_{K}$, we get $\rho_{\ell}(\sigma)$ is of finite order, and the trace condition implies that $\rho_{\ell}(\sigma)=1$. 
\end{proof}

\begin{definition}
\begin{enumerate}
\item
There exists a unique increasing filtration $\M_{r}(H^{2}_{\et}(X_{\overline{K}},\Q_{\ell}))$ on $H^{2}_{\et}(X_{\overline{K}},\Q_{\ell})$ such that $\M_{r}=0$ for $r \ll 0$, $\M_{r}= H^{2}_{\et}(X_{\overline{K}},\Q_{\ell})$ for $r \gg 0$, $N(\M_{r})\subset \M_{r-2}$ and $N^{r}$ induces an isomorphism $\gr_{r}^{\M}\simeq \gr_{-r}^{\M}$ for any positive integer $r$. We call $\M_{r}$ the \emph{monodromy filtration} on $H^{2}_{\et}(X_{\overline{K}},\Q_{\ell})$.
\item
If $X$ admits a strictly semi-stable model over $\oo_{K}$, we get the \emph{weight filtration} $\W_{r}(H^{2}_{\et}(X_{\overline{K}},\Q_{\ell}))$ on $H^{2}_{\et}(X_{\overline{K}},\Q_{\ell})$ by the weight spectral sequence (see \cite[Corollary 2.8]{Saito2003}). For general $X$, one can also define the weight filtration $\W_{r}$ by using de Jong's alteration. 
\end{enumerate}
\end{definition}

\begin{lemma}\phantomsection\label{matsumoto lindep}
\begin{enumerate}
\item
For any integer $r$, we have
\[
\M_{r}(H^{2}_{\et}(X_{\overline{K}},\Q_{\ell}))= \W_{r+2} (H^{2}_{\et}(X_{\overline{K}},\Q_{\ell})).
\]
\item
For any integer $r$, the dimension of $\gr^{\W}_{r}(H^{2}_{\et}(X_{\overline{K}},\Q_{\ell}))$ is independent of $\ell$. 
\end{enumerate}
\end{lemma}
\begin{proof}
The assertion $(1)$ is well-known as the weight monodromy conjecture for surfaces (see \cite[Satz 2.13]{Rapoport1982}, \cite[Lemma 3.9]{Saito2003}). The assertion $(2)$ follows from \cite[Theorem 3.3 (2)]{Matsumoto2016}.
\end{proof}

\subsubsection*{The proof of $(a)\Leftrightarrow (b)$ in Lemma \ref{lindep}}
Taking a completion, we may assume $K$ is complete.
We shall prove $(a) \Rightarrow (b)$.
Take prime numbers $\ell,\ell' \neq p$.
By \cite[Corollary 2.5]{Ochiai1999} (for imperfect residue fields, see \cite[Proposition 4.2]{Vidal2004}), we have $\tr(\rho_{\ell}(\sigma))=\tr(\rho_{\ell'}(\sigma))$ for any $\sigma\in I_{K}$. By the definition of the monodromy filtration, we have
\[
N_{\ell}= 0 \Leftrightarrow \dim (\gr^{M}_{0}(H^{2}_{\et}(X_{\overline{K}},\Q_{\ell}))) = \dim (H^{2}_{\et}(X_{\overline{K}},\Q_{\ell})).
\]

Therefore, by Lemma \ref{tr and monodromy} and Lemma \ref{matsumoto lindep}, we get the desired implication.

\section{Corollaries}
\subsection{Some remarks}
First, combining Theorem \ref{genunp} with Lemma \ref{lindep}, we obtain the main theorem in more generalized form.
\begin{theorem}\label{main}
Let $F$ be a finitely generated field over $\Q$, $R$ be a finite type algebra over $\Z$ which is a normal domain with the fraction field $F$, and $d$ be a positive integer. Then, the set
\begin{align*}
    S(F,R)\coloneqq
    \{ X 
    \mid 
    X \colon \textup{K3 surface over }F \textup{ satisfying the condition $(C)$}
    \}/F\textup{-isom}
\end{align*}
is finite.
Here, the condition \textup{$(C)$} is the following.

\begin{itemize}
\item[$(C)$]
For any height $1$ prime $\p \in \Spec R$, take a discrete valuation field $E_{\p}$ such that $E_{\p}$ is an algebraic extension of discrete valuation fields over $F=\Frac(R_{\p})$, the residue field of $E_{\p}$ is the perfection of the residue field of $R_{\p}$, and a uniformizer of $R_{\p}$ is also a uniformizer of $E_{\p}$. Then, there exists a prime number $\ell$ different from the residual characteristic of $\p$ such that 
$H^{2}_{\et}(X_{\overline{E_{\p}}},\Q_{\ell})$ is an unramified $\Gal(\overline{F}/E_{\p})$-representation. 
\end{itemize}
\end{theorem}
\begin{remark}\phantomsection\label{remmain}
\begin{enumerate}
\item
The field extension $E_{\p}$ in the condition $(C)$ always exists by \cite[Theorem 29.1]{Matsumura1989}.
\item
By Lemma \ref{lindep}, the unramifiedness assumption in the condition $(C)$ is independent of $\ell$. If the residual characteristic of $\p$ is positive, replacing $E_{\p}$ by the completion of it, we can replace this condition in terms of crystalline representations. 
\end{enumerate}
\end{remark}

\begin{proof}
Shrinking $\Spec R$ if necessary, we may assume that $R$ is smooth over $\Z$ since the generic fiber $R\otimes_{\Z}\Q$ is generically smooth over $\Q$.
Let $M$ be the order of $\GL_{22}(\F_{2})$. 
Shrinking $\Spec R$ again, we may assume that $1/M \in R$. 
Consider a height $1$ prime $\p \in \Spec R$, and we denote its residual characteristic by $p\geq 0$. 
Take an extension of valuation $\p$ to $\overline{F}$, and we denote it by $\overline{v}$. 
We denote the inertia subgroups by 
$I_{\overline{v}} \subset \Gal(\overline{F}/F), I'_{\overline{v}} \subset \Gal(\overline{F}/E_{\p}).$ 
We denote the $\Gal(\overline{F}/F)$-representation $H^{2}_{\et}(X_{\overline{K}},\Z_{2})$ 
by $\rho$. 
Then, by Remark \ref{remmain} (2), we get $\rho(I'_{\overline{v}})=1$. 
If $p=0$, we have $\rho(I_{\overline{v}}) = \rho(I'_{\overline{v}})=1$. 
If $p>0,$ for any finite index open normal subgroup $H$ of $\rho(I_{\overline{v}})$, we get $[\rho(I_{\overline{v}})\colon H]$ is a $p$-group. 
(Here, we use that for any finite extension of discrete valuation fields of characteristic $0$, the extension degree is equal to the product of the ramification index and the inertia degree).
Therefore we get 
\[
\rho(I_{\overline{v}})\cap (1+2\cdot \Mat_{22}(\Z_{2}))=1
\]
since the former is pro-$p$ and the latter is pro-$2$. 
Moreover, the image of $\rho(I_{\overline{v}})$ in $\GL_{22}(\F_{2})$ via the reduction map is trivial because $p$ does not divide $M$. 
Therefore, we get $\rho(I_{\overline{v}})=1$ even if $p>0$. 
Thus we have $S(F,R) \subset \Shaf(F,R)$, so $S(F,R)$ is a finite set.
\end{proof}

Next, as an immediate consequence of Theorem \ref{main}, we obtain the unpolarized Shafarevich conjecture for K3 surfaces over finitely generated fields of characteristic $0$.

\begin{definition}
Let $R_{\p}$ be a discrete valuation ring with maximal ideal $\p$, and $F$ be the fraction field of $R_{\p}$.
For a K3 surface $X$ over $F$,
we say $X$ has good reduction at $\p$ if there exists a smooth proper algebraic space over $R_{\p}$ whose generic fiber is isomorphic to $X$. 
Note that such model would be automatically a K3 family over $\Spec R_{\p}$ (see Definition \ref{K3} (2)). 
\end{definition}

\begin{corollary}\label{shafunp}
Let $F$ be a finitely generated field over $\Q$, and $R$ be a finite type algebra over $\Z$ which is a normal domain with the fraction field $F$.
Then, the set
\[
\left\{X \left|
\begin{array}{l}
X\colon \textup{K3 surface over }F,\\
X \textup{ has good reduction at any height }1 \textup{ prime ideal }\p \in \Spec R
\end{array}
\right.\right\}/F\textup{-isom}
\]
is finite.
\end{corollary}

\subsection{The finiteness of twists}
Here, we give the finiteness result of twists of K3 surfaces via a finite extension of characteristic $0$ fields.

\begin{corollary}\label{fintwists}
Let $F$ be a field of characteristic $0$, $E/F$ be a finite extension, and $X$ be a K3 surface over $F$. Then, the set
\[
\Tw_{E/F}(X)\coloneqq
\{Y\colon \textup{K3 surface over } F\mid Y_{E}\simeq_{E}X_{E}
\}/F\textup{-isom}
\]
is finite.
\end{corollary}

\begin{proof}
Clearly, we may assume $E/F$ is a finite Galois extension.
First, we will reduce the problem to the case of finitely generated fields.
Since $\Aut(X_{\overline{F}})$ is a finitely generated group (\cite[Proposition 2.2]{Sterk1985}), extending $E$ if necessary, we may assume $\Aut(X_{E}) =\Aut(X_{\overline{F}}).$
We can take a finitely generated field $E'\subset E$ on which $X$ and any elements of $\Aut(X_{E})$ are defined. 
Moreover, by extending $E'$ if necessary, we may assume $E'$ is $\Gal(E/F)$-stable and $\Gal(E/F)\rightarrow \Aut(E')$ is injective. 
Let $F'$ be the fixed subfield $E'^{\Gal(E/F)}$. Then, the description of twists
\begin{align*}
\Tw_{E/F}(X) &\simeq H^{1}(\Gal(E/F), \Aut(X_{E}))\\
&\simeq H^{1}(\Gal(E'/F'), \Aut(X_{E'})).
\end{align*}
implies that the desired finiteness is reduced to the case of $E'/F'$.

Thus, in the following of this proof, we assume $F$ is a finitely generated field and $E/F$ is a finite Galois extension. 
One can take a smooth proper morphism of schemes $\mathcal{X}\rightarrow \Spec R$ whose generic fiber is $X$, where $R$ is a smooth algebra over $\Z$ which is an integral domain with the fraction field $F$ and $1/2\in R$. Then, via a monodromy action, we get $H^{2}_{\et}(X_{\overline{F}},\Z_{2})$ is unramified over $\Spec R$. 
Let $\widetilde{R}$ be the normalization of $R$ in $E$. 
Shrinking $\Spec R$ if necessary, we may assume $\Spec \widetilde{R} \rightarrow \Spec R$ is a finite $\et$ale covering. Since $E$ is unramified over $\Spec R,$ by \cite[Proposition 3.3.6]{Fu2015}, we have
\[
\Ker(\pi_{1}(\Spec E, \overline{s})\rightarrow \pi_{1}(\Spec\widetilde{R},\overline{s}))
=\ker(\pi_{1}(\Spec F, \overline{s})\rightarrow \pi_{1}(\Spec R, \overline{s})).
\]
For any $Y\in \Tw_{E/F}(X)$, the isomorphism $Y_{E}\simeq_{E} X_{E}$ implies that the $\Gal(\overline{F}/E)$-action on $H^{2}_{\et}(Y_{\overline{F}},\Z_{2})$ arises from a $\pi_{1}(\Spec\widetilde{R},\overline{s})$-action. 
Moreover, because of the above equality, the $\Gal(\overline{F}/F)$-action on $H^{2}_{\et}(Y_{\overline{F}},\Z_{2})$ also arises from a $\pi_{1}(\Spec R,\overline{s})$-action. Hence we get a natural inclusion $\Tw_{E/F}(X)\hookrightarrow \Shaf (F, R),$ and thus the desired finiteness follows from Theorem \ref{genunp}.
\end{proof}


\begin{thebibliography}{FWG{\etalchar{+}}92}

\bibitem[And96]{Andre1996}
Y.~Andr\'e, \emph{On the {S}hafarevich and {T}ate conjectures for
  hyper-{K}\"ahler varieties}, Math. Ann. \textbf{305} (1996), no.~2, 205--248.
  \MR{1391213}

\bibitem[Bor63]{Borel1963}
A.~Borel, \emph{Arithmetic properties of linear algebraic groups}, Proc.
  {I}nternat. {C}ongr. {M}athematicians ({S}tockholm, 1962), Inst.
  Mittag-Leffler, Djursholm, 1963, pp.~10--22. \MR{0175901}

\bibitem[Cas82]{Cassels1982}
J.~W.~S. Cassels, \emph{Rational quadratic forms}, Proceedings of the
  {I}nternational {M}athematical {C}onference, {S}ingapore 1981 ({S}ingapore,
  1981), North-Holland Math. Stud., vol.~74, North-Holland, Amsterdam-New York,
  1982, pp.~9--26. \MR{690091}

\bibitem[CI99]{Coleman1999}
R.~Coleman and A.~Iovita, \emph{The {F}robenius and monodromy operators for
  curves and abelian varieties}, Duke Math. J. \textbf{97} (1999), no.~1,
  171--215. \MR{1682268}

\bibitem[CLL19]{Chiarellotto2019}
B.~Chiarellotto, C.~Lazda, and C.~Liedtke, \emph{A
  {N}\'{e}ron-{O}gg-{S}hafarevich criterion for {K}3 surfaces}, Proc. Lond.
  Math. Soc. (3) \textbf{119} (2019), no.~2, 469--514. \MR{3959051}

\bibitem[Fal83]{Faltings1983}
G.~Faltings, \emph{Endlichkeitss\"atze f\"ur abelsche {V}ariet\"aten \"uber
  {Z}ahlk\"orpern}, Invent. Math. \textbf{73} (1983), no.~3, 349--366.
  \MR{718935}

\bibitem[Fu15]{Fu2015}
L.~Fu, \emph{Etale cohomology theory}, revised ed., Nankai Tracts in
  Mathematics, vol.~14, World Scientific Publishing Co. Pte. Ltd., Hackensack,
  NJ, 2015. \MR{3380806}

\bibitem[FWG{\etalchar{+}}92]{Faltings1992}
G.~Faltings, G.~W\"ustholz, F.~Grunewald, N.~Schappacher, and U.~Stuhler,
  \emph{Rational points}, third ed., Aspects of Mathematics, E6, Friedr. Vieweg
  \& Sohn, Braunschweig, 1992, Papers from the seminar held at the
  Max-Planck-Institut f\"ur Mathematik, Bonn/Wuppertal, 1983/1984, With an
  appendix by W\"ustholz. \MR{1175627}

\bibitem[GRR72]{SGA7-I}
A.~Grothendieck, M.~Raynaud, and D.~S. Rim, \emph{Groupes de monodromie en
  g\'eom\'etrie alg\'ebrique. {I}}, Lecture Notes in Mathematics, Vol. 288,
  Springer-Verlag, 1972, S{\'e}minaire de G{\'e}om{\'e}trie Alg{\'e}brique du
  Bois-Marie 1967--1969 (SGA 7 I).

\bibitem[HH09]{Harada2009}
S.~Harada and T.~Hiranouchi, \emph{Smallness of fundamental groups for
  arithmetic schemes}, J. Number Theory \textbf{129} (2009), no.~11,
  2702--2712. \MR{2549526}

\bibitem[Huy16]{Huybrechts2016}
D.~Huybrechts, \emph{Lectures on {K}3 surfaces}, Cambridge Studies in Advanced
  Mathematics, vol. 158, Cambridge University Press, Cambridge, 2016.
  \MR{3586372}

\bibitem[IIK18]{Ito2018}
K.~Ito, T.~Ito, and T.~Koshikawa, \emph{{CM liftings of K3 surfaces over finite
  fields and their applications to the Tate conjecture}}, preprint (2018),
  arXiv:1809.09604.

\bibitem[IM20]{Imai2020}
N.~Imai and Y.~Mieda, \emph{Potentially good reduction loci of {S}himura
  varieties}, Tunis. J. Math. \textbf{2} (2020), no.~2, 399--454. \MR{3990825}

\bibitem[LM18]{Liedtke2018}
C.~Liedtke and Y.~Matsumoto, \emph{Good reduction of {K}3 surfaces}, Compos.
  Math. \textbf{154} (2018), no.~1, 1--35.

\bibitem[Mat89]{Matsumura1989}
H.~Matsumura, \emph{Commutative ring theory}, second ed., Cambridge Studies in
  Advanced Mathematics, vol.~8, Cambridge University Press, Cambridge, 1989,
  Translated from the Japanese by M. Reid. \MR{1011461}

\bibitem[Mat15]{Matsumoto2015a}
Y.~Matsumoto, \emph{Good reduction criterion for {K}3 surfaces}, Math. Z.
  \textbf{279} (2015), no.~1-2, 241--266. \MR{3299851}

\bibitem[{Mat}16]{Matsumoto2016}
Y.~{Matsumoto}, \emph{{Degeneration of K3 surfaces with non-symplectic
  automorphisms}}, preprint (2016), arXiv:1612.07569.

\bibitem[Mil90]{Milne1990}
J.~S. Milne, \emph{Canonical models of (mixed) {S}himura varieties and
  automorphic vector bundles}, Automorphic forms, {S}himura varieties, and
  {$L$}-functions, {V}ol.\ {I} ({A}nn {A}rbor, {MI}, 1988), Perspect. Math.,
  vol.~10, Academic Press, Boston, MA, 1990, pp.~283--414. \MR{1044823}

\bibitem[MP15]{MadapusiPera2015}
K.~Madapusi~Pera, \emph{The {T}ate conjecture for {K}3 surfaces in odd
  characteristic}, Invent. Math. \textbf{201} (2015), no.~2, 625--668.
  \MR{3370622}

\bibitem[MP16]{MadapusiPera2016}
\bysame, \emph{Integral canonical models for spin {S}himura varieties}, Compos.
  Math. \textbf{152} (2016), no.~4, 769--824. \MR{3484114}

\bibitem[Nik79]{Nikulin1979}
V.~V. Nikulin, \emph{Integer symmetric bilinear forms and some of their
  geometric applications}, Izv. Akad. Nauk SSSR Ser. Mat. \textbf{43} (1979),
  no.~1, 111--177, 238. \MR{525944}

\bibitem[Och99]{Ochiai1999}
T.~Ochiai, \emph{{$l$}-independence of the trace of monodromy}, Math. Ann.
  \textbf{315} (1999), no.~2, 321--340. \MR{1715253}

\bibitem[Riz06]{Rizov2006}
J.~Rizov, \emph{Moduli stacks of polarized {$K3$} surfaces in mixed
  characteristic}, Serdica Math. J. \textbf{32} (2006), no.~2-3, 131--178.
  \MR{2263236}

\bibitem[Riz10]{Rizov2010}
\bysame, \emph{Kuga-{S}atake abelian varieties of {K}3 surfaces in mixed
  characteristic}, J. Reine Angew. Math. \textbf{648} (2010), 13--67.
  \MR{2774304}

\bibitem[RZ82]{Rapoport1982}
M.~Rapoport and T.~Zink, \emph{\"{U}ber die lokale {Z}etafunktion von
  {S}himuravariet\"{a}ten. {M}onodromiefiltration und verschwindende {Z}yklen
  in ungleicher {C}harakteristik}, Invent. Math. \textbf{68} (1982), no.~1,
  21--101. \MR{666636}

\bibitem[Sai03]{Saito2003}
T.~Saito, \emph{Weight spectral sequences and independence of {$l$}}, J. Inst.
  Math. Jussieu \textbf{2} (2003), no.~4, 583--634. \MR{2006800}

\bibitem[She17]{She2017}
Y.~She, \emph{{The unpolarized Shafarevich Conjecture for K3 Surfaces}},
  preprint (2017), arXiv:1705.09038.

\bibitem[Ste85]{Sterk1985}
H.~Sterk, \emph{Finiteness results for algebraic {$K3$} surfaces}, Math. Z.
  \textbf{189} (1985), no.~4, 507--513. \MR{786280}

\bibitem[Tat94]{Tate1994}
J.~Tate, \emph{Conjectures on algebraic cycles in {$l$}-adic cohomology},
  Motives ({S}eattle, {WA}, 1991), Proc. Sympos. Pure Math., vol.~55, Amer.
  Math. Soc., Providence, RI, 1994, pp.~71--83. \MR{1265523}

\bibitem[Vid04]{Vidal2004}
I.~Vidal, \emph{Th\'{e}orie de {B}rauer et conducteur de {S}wan}, J. Algebraic
  Geom. \textbf{13} (2004), no.~2, 349--391. \MR{2047703}

\bibitem[Zar85]{Zarhin1985}
Y.~G. Zarhin, \emph{A finiteness theorem for unpolarized abelian varieties over
  number fields with prescribed places of bad reduction}, Invent. Math.
  \textbf{79} (1985), no.~2, 309--321. \MR{778130}

\end{thebibliography}

\newcommand{\etalchar}[1]{$^{#1}$}
\providecommand{\bysame}{\leavevmode\hbox to3em{\hrulefill}\thinspace}
\providecommand{\MR}{\relax\ifhmode\unskip\space\fi MR }
\providecommand{\MRhref}[2]{%
  \href{http://www.ams.org/mathscinet-getitem?mr=#1}{#2}
}
\providecommand{\href}[2]{#2}

\end{document}